\documentclass[12 pt, oneside]{amsart}
  \usepackage[utf8]{inputenc}
  \usepackage[margin=3cm]{geometry}
\makeatletter
\def\l@subsection{\@tocline{1}{0pt}{2pc}{1pc}{}}
\def\l@subsubsection{\@tocline{2}{0pt}{2pc}{1pc}{}}
\makeatother

  \usepackage{stmaryrd}
\usepackage{amsmath,amssymb,amsthm, mathrsfs, mathtools, bbold, mathbbol}
\usepackage{amscd,mathtools}
\usepackage{comment}
  \usepackage{geometry}
  \usepackage{graphicx,epstopdf}
  \usepackage{color,xcolor}
  \usepackage{enumerate}
  \usepackage{xspace}
  \usepackage{tipa}
  \usepackage[user,titleref]{zref}
 \usepackage{soul}
 \usepackage{xcolor}
  \usepackage{epsfig}

\usepackage{tikz}
\usetikzlibrary{shapes.geometric}
\usetikzlibrary{angles}

\usepackage{caption, subcaption}
\usepackage{tikz-cd}
\usetikzlibrary{calc,decorations.pathreplacing}
\usepackage{tikz}
\usetikzlibrary{calc}
\usetikzlibrary{arrows}
\usetikzlibrary{decorations.pathreplacing}
\usetikzlibrary{intersections}

\usepackage{mathrsfs}
\usetikzlibrary{matrix}
\usetikzlibrary{positioning}
\DeclareMathAlphabet{\pazocal}{OMS}{zplm}{m}{n}
\tikzset{>=stealth}
  \usepackage{hyperref}

  \hypersetup{
    colorlinks=true,
    linkcolor=blue,
    filecolor=magenta,      
    urlcolor=cyan,
    citecolor=purple,
    pdftitle={Overleaf Example},
    pdfpagemode=FullScreen,
    }

\usepackage{float}

  \newcommand{\calB}{\mathcal{B}}

  \newcommand{\calN}{\mathcal{N}}

  \newcommand{\calU}{\mathcal{U}}
    \newcommand{\calV}{\mathcal{V}}

  \newcommand{\NN}{\mathbb{N}}

  \newcommand{\ZZ}{\mathbb{Z}}

  \newcommand{\bfa}{\textbf{a}}
  \newcommand{\bfb}{\textbf{b}}

  \newcommand{\bfz}{\textbf{z}}  

  \newcommand{\gothic}{\mathfrak}
  \newcommand{\Ga}{{\gothic a}}
  \newcommand{\Gb}{{\gothic b}}
  \newcommand{\Gc}{{\gothic c}}

  \newcommand{\go}{{\gothic o}}

  \newtheorem{theorem}{Theorem}[section]
  \newtheorem{proposition}[theorem]{Proposition}
  \newtheorem{corollary}[theorem]{Corollary}
  \newtheorem{lemma}[theorem]{Lemma}
  \newtheorem{conjecture}[theorem]{Conjecture}

  \newtheorem*{claim*}{Claim}

  \newtheorem{introthm}{Theorem}
  \newtheorem{introcor}[introthm]{Corollary}

  \theoremstyle{definition}
  \newtheorem{definition}[theorem]{Definition}

  \newtheorem*{question*}{Question}
  \newtheorem*{answer*}{Answer}
  \newtheorem*{application*}{Application}
  \newtheorem{notation}[theorem]{Notation}

  \theoremstyle{remark}
  \newtheorem{remark}[theorem]{Remark}
  \newtheorem*{remark*}{Remark}
  
  \newcommand{\defref}[1]{Definition~\ref{#1}}

  \newcommand{\pka}{\partial_{\kappa}}

  \newcommand{\Teich}{{Teichm\"uller }} % A Nazi dude
  \newcommand{\Ham}{{Hamenst\"adt }}

  \newcommand{\sQ}{{\sf Q}}

  \renewcommand{\bfa}{{\sf a}}

  \newcommand{\qq}{{\sf q}}   
  \newcommand{\rr}{{\sf r}}

  \newcommand{\param}{{\mathchoice{\mkern1mu\mbox{\raise2.2pt\hbox{$
  \centerdot$}}
  \mkern1mu}{\mkern1mu\mbox{\raise2.2pt\hbox{$\centerdot$}}\mkern1mu}{
  \mkern1.5mu\centerdot\mkern1.5mu}{\mkern1.5mu\centerdot\mkern1.5mu}}}
\DeclarePairedDelimiterX{\norm}[1]{\lvert}{\rvert}{#1}
\DeclarePairedDelimiterX{\Norm}[1]{\lVert}{\rVert}{#1}

  \renewcommand{\setminus}{{\smallsetminus}}

  \newcommand{\ST}{\mathbin{\Big|}} 
  \newcommand{\from}{\colon\thinspace}

\newcommand{\K}{\kappa}

\newcommand{\ob}{\mathfrak{o}}

\newcommand{\CAT}{\ensuremath{\operatorname{CAT}(0)}\xspace}         
    
\title[Properties of the QR boundary]{Topological and dynamic properties of the sublinearly Morse boundary and the quasi-redirecting boundary}
 \author{Jacob Garcia}
 \address{Department of Mathematical Sciences,  Smith College, Northampton, MA, USA}
 \email{jgarcia46@smith.edu}
 
  \author{Yulan Qing}
 \address{Department of Mathematics,  University of Tennessee at Knoxville, Knoxville, TN, USA}
 \email{yqing@utk.edu}
 
  \author{Elliott Vest}
 \address{Department of Mathematics,  University of California Riverside, Riverside, CA, USA}
 \email{elliott.vest@email.ucr.edu}

\begin{document}
\begin{abstract}
Sublinearly Morse boundaries of proper geodesic spaces are introduced by Qing, Rafi and Tiozzo. Expanding on this work, Qing and Rafi recently developed the quasi-redirecting boundary, denoted $\partial G$, to include all directions of metric spaces at infinity. Both boundaries are topological spaces that consist of equivalence classes of quasi-geodesic rays and are quasi-isometrically invariant. In this paper, we study these boundaries when the space is equipped with a geometric group action. In particular, we show that $G$ acts minimally on $\pka G$ and that contracting elements of G induces a weak north-south dynamic on $\pka G$. We also prove, when $\partial G$ exists and $|\pka G|\geq3$, $G$ acts minimally on $\partial G$ and $\partial G$ is a second countable topological space. The last section concerns the restriction to proper \CAT spaces and finite dimensional \CAT cube complexes. We show that when $G$ acts geometrically on a finite dimensional \CAT cube complex (whose QR boundary is assumed to exist), then a nontrivial QR boundary implies the existence of a Morse element in $G$. Lastly, we show that if $X$ is a proper cocompact \CAT space, then $\partial G$ is a visibility space. 

\end{abstract}

\maketitle
\setcounter{tocdepth}{1}
 \tableofcontents
 
\section{Introduction}
Gromov introduced gromov hyperbolicity and hyperbolic groups in \cite{gromov}. He also introduced a compactification of gromov hyperbolic spaces call the Gromov boundary. A central fact connecting these concepts is that Gromov hyperbolicity is a quasi-isometry invariant: quasi-isometries extend equivariantly to homeomorphisms on the Gromov boundary. In the years after, the Gromov boundary has proven to be of central importance in understanding hyperbolic groups and hyperbolic manifolds, see \cite{kapovich2002boundaries} for a survey of these results. In recent decades, various boundaries have been constructed for non-hyperbolic groups to extend the program started by Gromov. In particular, Rafi-Qing and Tiozzo (\cite{QRT1}, \cite{QRT2}) defined the sublinearly Morse boundary, including geodesic rays whose Morse-ness can decay sublinearly with distance from the base point. These boundaries are the first %group-invariants and 
% -- Cordes' Morse Boundary is Group invariant, shown in [Cordes, Proposition 3.7]
metrizable topological spaces for all finitely generated groups. Their hyperbolic features are discussed in works such as \cite{IZ24}, \cite{MQZ20}, \cite{QZ}. These studies center around \CAT spaces and \CAT cube complexes. In the first part of this paper, we continue to study their hyperbolic-like features in the more general setting of proper geodesic spaces.

\subsection{Minimality of the group action} 
A group is said to act minimally on a topological space if every orbit is a dense subset of the space. We show that this property is enjoyed by $\kappa$-boundaries. In contrast with the identifications with Poisson boundaries in various settings, the minimality result is evidence to the fact that the boundary is not too large in excess of the orbit of a point under the group action. The first result fully generalizes the same property shown in \CAT spaces in \cite[Theorem 3.3]{QZ}.
\begin{introthm}\label{thm:minimality of sublinear}
Every finitely generated group $G$ acts minimally on $\pka G$. That is, if $\Ga \in \pka G$ is any element in $\pka G$, we have that $G \cdot \Ga $ is a dense subset of $\pka G$.
\end{introthm}

Aiming to understand all directions up to quasi-isometry, and to compactify the sublinearly Morse boundary, Rafi-Qing recently introduced a new boundary for metric spaces called the \emph{quasi-redirecting boundary}, or the QR boundary for short \cite{QR24}. The QR boundary expands the sublinearly Morse boundary topologically and is often compact, which is one of its key advantages.

Here is the main idea of the QR boundary: let $\alpha, \beta \colon [0, \infty) \to X$ be two quasi-geodesic rays in a  metric space $X$. We say $\alpha$ can be \textit{quasi-redirected} to $\beta$ (and write $\alpha \preceq \beta$) if there exists a pair of constant $(q, Q)$ such that for every $r >0$, there exists a $(q, Q)$--quasi-geodesic ray $\gamma$ that is identical to $\alpha$ 
%inside the ball $B(\alpha(0), r)$ 
up to distance $r$, and eventually $\gamma$ becomes identical to $\beta$. We say $\alpha \simeq \beta$ if $\alpha \preceq \beta$ and $\beta \preceq \alpha$. The resulting set of equivalence classes forms a poset, denoted by $P(X)$. The poset $P(X)$ equipped with a ``cone-like topology'' is called the \textit{quasi-redirecting boundary} (QR boundary) of $X$ and we denote it by $\partial X$.

\begin{figure}[H]
\centering
\begin{tikzpicture}[scale=0.7] 
\draw[blue, thick] (0,0) -- (9,0) node[below] {$\beta$};
\draw[red, thick] (0,0) -- (0,4) node[left] {$\alpha$};
\draw[|-|, thin] (-0.4,0)--(-0.4,1.8);
\node[left] at (-0.4,0.9) {$r$};
\node at (2,2) {$\gamma$};

 \pgfsetlinewidth{1pt}
 \pgfsetplottension{.75}
 \pgfplothandlercurveto
 \pgfplotstreamstart
 \pgfplotstreampoint{\pgfpoint{0 cm}{0cm}}  
 \pgfplotstreampoint{\pgfpoint{.3 cm}{3 cm}}   
 \pgfplotstreampoint{\pgfpoint{2 cm}{1.4 cm}}
 \pgfplotstreampoint{\pgfpoint{3 cm}{.7 cm}}
 \pgfplotstreampoint{\pgfpoint{4.5 cm}{.3 cm}}
 \pgfplotstreampoint{\pgfpoint{6 cm}{.1 cm}} 
 \pgfplotstreampoint{\pgfpoint{8 cm}{.005 cm}} 
 \pgfplotstreampoint{\pgfpoint{9 cm}{0 cm}} 
 \pgfplotstreamend
 \pgfusepath{stroke} 
 \end{tikzpicture}
\caption{The ray $\alpha$ can be quasi-redirected to $\beta$ at radius $r$.}
\end{figure}
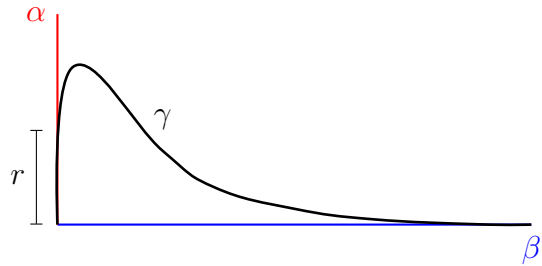
\begin{introthm}\label{introB}
  The sublinearly Morse boundary is a dense subset of the QR boundary. Assume that $\pka G$ exists and $|\pka G|\geq3$. Let $\bfb$ be any element of $\partial G$, then there exists an infinite sequence $\{\Ga_i\} \in \pka G$ such that the sequence converges to $\bfb$ in the topology of $\partial G$. 
\end{introthm}

 In \cite{QR24}, it is shown that the QR boundary is second countable when the space in question is an asymptotically tree graded spaces with mono-directional subspace. A consequence of Theorem~\ref{introB} is that we now obtain second countability for all quasi-redirecting boundaries when exists.

\begin{introcor}
  Assume that $\partial G$ exists.  If $|\pka G|\geq 3$, then $\partial G$ is second countable. 
\end{introcor}

However, the group $G$ does not always act minimally on $\partial_{QR} G$ as discussed in Section~\ref{counter}. There are examples where there exists points in $\partial_{QR} G$ whose orbit is not dense in $\partial_{QR} G$. 

\subsection{Morse elements and visibility} Question 4.4 in \cite{QR24} asks that if $G$ does not have an Morse element, is $P(G)$ a single point. In this paper we answer the question in the affirmative for the setting of finite dimensional \CAT cube complexes. 
\begin{introthm}
If $G$ acts geometrically on a finite dimensional \CAT cube complex and  $|P(G)|\geq2$, then $G$ contains a Morse element.
\end{introthm}

As a middle step of establishing this result we also obtain visibility of $P(X)$. Roughly speaking, a set of directions is a visibility space if there is a bi-infinite geodesic line connecting every pair of directions.  The Gromov boundary is a visibility space and this property helps to understand the connection between quasiconformal maps on the boundary and quasi-isometries on the space. Furthermore, the visibility property is not only true for hyperbolic groups. A. Karlsson in \cite{Karlsson} proved that it is true on the Floyd boundary. Visibility also holds for Morse Boundary\cite{CCM19}, sublinearly Morse boundaries \cite{DZ}, as well as the Bowditch boundary for relatively hyperbolic groups and Floyd compacitification.  Visibility finds applications in studying random walks on countable groups \cite{Tiozzo} and connecting the Floyd boundary to the Bowditch boundary \cite{Victor} in relatively hyperbolic settings. It is conceivable that some of these applications can be extended to quasi-redirecting boundary when $\partial X$ exists.

\begin{introcor}
Let $X$ be a proper \CAT metric space with a cocompact action. Then $P(X)$ is a visibility space.
\end{introcor}

\subsection*{History}  In the classical setting, elements that act on the circle with a single attracting point and a single repelling point are known as having north-south dynamics. These actions play important roles in the dynamics of actions on the circle, see for example Thurston \cite{Thurstoncircle}.  These features persists in the non-hyperbolic group setting and has contributed to the study of Out($F_n$) \cite{caglarclay}, and Thompsons group \cite{CT21}, to name a few. Minimality of group actions is among the basic topological and dynamical properties established for the Gromov boundaries in \cite{gromov}, and the property of minimality has been a key ingredient in proving north-south dynamics \cite{BK02} \cite{liu2021dynamics}, \cite{QZ}.

In non-hyperbolic-group settings where the group has some weaker hyperbolic-like 
properties, the sublinearly Morse boundaries and the quasi-redirecting boundaries are preceded by geometric constructions aiming to generalize the Gromov boundary.
The first of such constructions were created by Charney and Sultan in 2013 \cite{contracting}. Their \emph{contracting boundary} of \CAT spaces was shown to be a first quasi-isometrically invariant geometric boundary in 
non-hyperbolic settings. This construction was generalized to the proper geodesic setting by Cordes' creation of the Morse boundary in \cite{Morse}. Morse boundaries are equipped with a \emph{direct limit topology} and  are invariant under 
quasi-isometries. However, the Morse boundary is frequently not second countable, and in general, sample paths of simple random walks on groups do not converge to points in the Morse boundary.

In comparison, the sublinearly Morse boundary is frequently a topological model (and hence a group invariant topological model) for suitable randoms walks on the associated group. For example,  right-angled Artin groups and mapping class groups can be identified with the Poisson boundary of associated random walks \cite{QRT1,QRT2}. Meanwhile, genericity of a more geometric flavor is also exhibited for sublinearly Morse boundaries. In \cite{GQR22}, genericity of sublinearly Morse directions under Patterson Sullivan measure was shown to hold in the more general context of actions which admit a strongly contracting element. In fact, the results in \cite{GQR22} concerning stationary measures were recently claimed in a different setting by Inhyeok Choi \cite{Choi}, who in place of ergodic theoretic and boundary techniques uses a pivoting technique developed by Gou\"ezel\cite{Gou22}. Also, following  \cite{wenyuan}, genericity of sublinearly Morse directions on the horofunction boundary was recently shown for all proper statistically convex-cocompact actions on proper metric spaces \cite{QY24}.

\subsection{Acknowledgement} The authors would like to thank the Erwin Schr\:odinger International Institute for Mathematics and Physics (ESI) for hosting the 2023 Geometric and Asymptotic Group Theory with Applications (GAGTA) conference, where this paper was first discussed and formed. 

\section{Preliminaries in boundary constructions}
\subsection{Quasi-geodesic metric space and surgeries}
In this section, we establish some basic definitions and notations and recall some 
surgery lemmas between quasi-geodesics. 

%\subsection*{Quasi-Isometry and Quasi-geodesic rays}
\begin{definition}[Quasi-Isometric embedding] \label{Def:Quasi-Isometry} 
Let $(X , d_X)$ and $(Y , d_Y)$ be metric spaces. For constants $q \geq 1$ and
$Q \geq 0$, we say a map $\Phi \from X \to Y$ is a 
$(q, Q)$--\textit{quasi-isometric embedding} if, for all $x_1, x_2 \in X$,
$$
\frac{1}{q} d_X (x_1, x_2) - Q  \leq d_Y \big(\Phi (x_1), \Phi (x_2)\big) 
   \leq q \, d_X (x_1, x_2) + Q.
$$
If, in addition, every point in $Y$ lies in the $Q$--neighbourhood of the image of 
$\Phi$, then $\Phi$ is called a $(q, Q)$--\emph{quasi-isometry}. This is equivalent to saying
that $\Phi$ has a \emph{quasi-inverse}. That is, there exists constants $q', Q'>0$ and a $(q',Q')$--quasi-isometric 
embedding $\Psi \from Y \to X$ such that,
\[
\forall x \in X \quad d_X\big(x, \Psi\Phi(x)\big) \leq Q' \qquad\text{and}\qquad
\forall y \in Y \quad d_Y\big(y, \Phi\,\Psi(y)\big) \leq Q'.
\]
When such a map $\Phi$ exists, we say $(X, d_X)$ and $(Y, d_Y)$ are \textit{quasi-isometric}. 
\end{definition}

\begin{definition}[Quasi-Geodesics] \label{Def:Quadi-Geodesic} 
A \emph{quasi-geodesic} in a metric space $X$ is a quasi-isometric embedding $\alpha \from I \to X$ where $I \subset \mathbb{R}$ is an 
(possibly infinite) interval. That is, $\alpha \from I \to X$ is a $(q,Q)$--quasi-geodesic if, for all $s, t \in I$, we have 
\[
\frac{|t-s|}{q} - Q  \leq d_X \big(\alpha(s), \alpha(t)\big)  \leq q \cdot |s-t| +Q. 
\]
Furthermore, in this paper we always assume $\alpha$ is $(2q+2Q)$--Lipschitz, and hence, $\alpha$ is continuous. By \cite[Lemma 2.3]{QR24} the Lipschitz assumption can be made without loss of generality. 
\end{definition}

\begin{notation} 
To simplify notation, we use $\qq=(q, Q) \in [1, \infty) \times [0, \infty)$ to indicate a pair of  constants.
That is, we say $\Phi \from X \to Y$ is a $\qq$--quasi-isometry, and $\alpha$ is a 
$\qq$-quasi-geodesic. We fix a base point $\go$ in the metric space $X$. By a \emph{$\qq$--ray} 
we mean a $\qq$--quasi-geodesic  $\alpha \from [0, \infty) \to X$ such that $\alpha(0) = \go$. We shall often refer to the image of $\alpha$ simply as $\alpha$, e.g. $x\in\alpha$ as opposed to $x\in im(\alpha)$.

For $r>0$, let $B_r \subset X$ is the open ball of radius $r$ centered at $\go$ and let $B^{c}_{r}= X - B_r$, the complement of $B_r$ in $X$. 
For a $\qq$--ray $\alpha$ and $r>0$, we let $t_r \geq 0 $ denote the first time when $\alpha$ first intersects 
$B^{c}_{r}$. We denote $\alpha(t_r)$ by $\alpha_{r} \in X$.  Also, let 
\[
\alpha|_r =\alpha[0,t_r]
\] 
be the restriction $\alpha$ to the interval $[0, t_r]$, that is, the quasi geodesic segment of $\alpha$ from $\go$ to $\alpha_r$. 
In general, for an interval $I \subset [0, \infty)$, $\alpha|_I$ denotes the restriction of $\alpha$ to $I$. In addition, we use 
$\alpha|_{r, \infty}$ to denote the tail of $\alpha$ from the point when $\alpha$ \emph{last} exit the ball of radius $r$. Given two points $x,y\in\alpha$, we denote the restriction of $\alpha$ between these points as $[x,y]_\alpha$.

We also use $d(\param, \param)$ instead of $d_X(\param, \param)$ when the metric space $X$ is fixed.
For $x \in X$, $\Norm{x}$ denotes $d(\go, x)$. 
\end{notation} 

\subsection*{Assumption 0}\label{assump0} (No dead ends) The metric space $X$ is always assumed to be a proper, geodesic metric space. 
Furthermore, there exist a pair of constants $\qq_0$ such that every point $x \in X$ lies on an infinite $\qq_0$-quasi-geodesic ray.  

Recall that every proper quasi-geodesic metric space is quasi-isometric to a proper geodesic metric space 
(see for example \cite[Proposition 5.3.9]{Clara}). So the first condition 
in the Assumption 0 is not a strong assumption. Furthermore, Cayley graphs of all finitely generated groups satisfies Assumption 0, \cite[Lemma 2.5]{QR24}.

\subsection*{Surgery between quasi-geodesics}
In this section we present several methods to produce a quasi-geodesic as a concatenation 
of other geodesics or quasi-geodesics. The statements are intuitively clear and the proofs are elementary. 
So, this section should probably be skipped on the first reading of the paper. 
First, we recall a few surgery lemmas from \cite{QRT1} and \cite{QRT2}. 

\begin{lemma}\label{surgery}
Let $X$ be a metric space satisfying Assumption 0. 
The following statements are found in \cite[Lemma 2.5]{QRT1} and \cite[Lemma 3.8]{QRT2}. 
 
\begin{itemize}
\item\label{concate}
Consider a point $x \in X$ and a $(q, Q)$--quasi-geodesic segment $\beta$ 
connecting a point $z \in X$ to a point $w \in X$. Let $y$ be a closest point in $\beta$
to $x$. Then 
\[
\gamma = [x, y] \cup [y, z]_\beta
\] 
is a $(3q, Q)$--quasi-geodesic.

\begin{figure}[H]
\begin{tikzpicture}[scale=0.9]
 \tikzstyle{vertex} =[circle,draw,fill=black,thick, inner sep=0pt,minimum size=.5 mm]
[thick, 
    scale=1,
    vertex/.style={circle,draw,fill=black,thick,
                   inner sep=0pt,minimum size= .5 mm},
                  
      trans/.style={thick,->, shorten >=6pt,shorten <=6pt,>=stealth},
   ]

  \node[vertex] (z) at (0,0)[label=left:$z$] {}; 
  \node[vertex] (w) at (8,0) [label=right:$w$]  {}; 
  \node[vertex] (x) at (4,2) [label=right:$x$]  {};     
  \node[vertex] (y) at (4,0) [label=below:$y$]  {};     
  \node (a) at (.9,.8) [label=right:$\beta$]  {};    
  \draw[thick, dashed]  (4, 2)--(4, 0){};
 
  \pgfsetlinewidth{1pt}
  \pgfsetplottension{.75}
  \pgfplothandlercurveto
  \pgfplotstreamstart
  \pgfplotstreampoint{\pgfpoint{0cm}{0cm}}  
  \pgfplotstreampoint{\pgfpoint{1.4cm}{-.6cm}}   
  \pgfplotstreampoint{\pgfpoint{1.3cm}{.2cm}}
  \pgfplotstreampoint{\pgfpoint{3cm}{-.4cm}}
  \pgfplotstreampoint{\pgfpoint{4cm}{0cm}}
  \pgfplotstreampoint{\pgfpoint{5cm}{-.2cm}}
  \pgfplotstreampoint{\pgfpoint{6cm}{.3cm}}
  \pgfplotstreampoint{\pgfpoint{7cm}{-.7cm}}
  \pgfplotstreampoint{\pgfpoint{8cm}{0cm}}
  \pgfplotstreamend
  \pgfusepath{stroke} 
  \end{tikzpicture}
  
\caption{The concatenation of the geodesic segment $[x,y]$ 
and the quasi-geodesic segment $[y,z]_\beta$ is a quasi-geodesic.}
\end{figure}
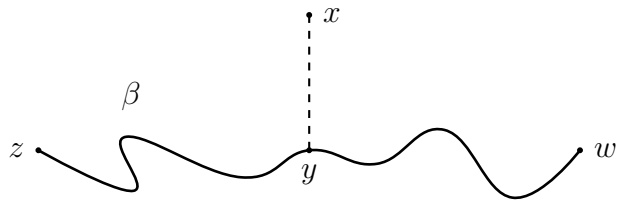

\item\label{redirect11} Let $\beta$ be a geodesic ray and 
$\gamma$ be a $(q, Q)$--ray. 
For $r>0$, assume that $d(\beta_\rr, \gamma)\leq \rr/2$. Then, there exists a
$(9q,Q)$--quasi-geodesic $\gamma'$ where $\gamma'(t) =\beta(t)$ for large values of $t$ and 
\[
\gamma|_{r/2} = \gamma'|_{r/2}. 
\]
\item
Consider a $(\qq, \sQ)$--quasi-geodesic ray  $\alpha \from [0, \infty) \to X$ and a finite 
$(\qq, \sQ)$--quasi-geodesic segment $\beta \from [a,b] \to X$. Then there is 
$s_0 \in [0, \infty)$ such that the following holds: for $ s \in [s_0, \infty)$ let 
$s_\gamma \in [s, \infty)$ and $t_\gamma \in [a,b]$ be such that
$[\beta(t_\gamma), \alpha(s_\gamma)]$ is a geodesic segment that realizes the set distance between 
$\alpha[s, \infty)$ and $\beta$. Then  
\[
\gamma = \beta[a, t_\gamma] \cup [\beta(t_\gamma), \alpha(s_\gamma)] \cup \alpha[s_\gamma, \infty) 
\] 
is a $(4\qq, 3\sQ)$--quasi-geodesic.

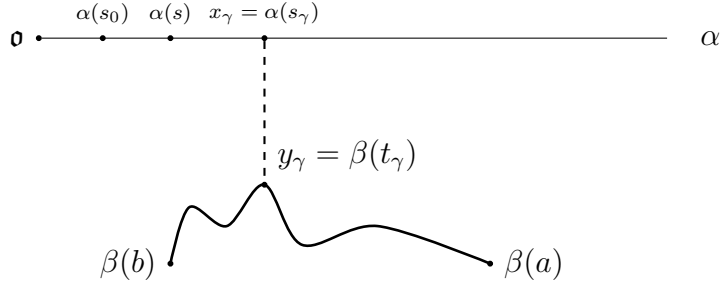
\begin{figure}[H]
\begin{tikzpicture}[scale=0.5]
 \tikzstyle{vertex} =[circle,draw,fill=black,thick, inner sep=0pt,minimum size=.5 mm]
[thick, 
    scale=1,
    vertex/.style={circle,draw,fill=black,thick,
                   inner sep=0pt,minimum size= .5 mm},
                  
      trans/.style={thick,->, shorten >=6pt,shorten <=6pt,>=stealth},
   ]

  \node[vertex] (o) at (-3,5)[label=left:$\go$] {}; 
  \node[vertex]  at (-1.3,5)[label=above:\tiny $\alpha(s_0)$] {}; 
  \node[vertex]  at (.5,5)[label=above:\tiny $\alpha(s)$] {}; 
  \node(a) at (14,5)[label=right:$\alpha$] {}; 
  \draw (o)--(a){};

 % \node[vertex] at (9,5) [label=above:${x=\alpha(s_x)}$] {};    
  \node[vertex](c) at(3, 5)[label=above:\tiny ${x_\gamma=\alpha(s_\gamma)}$]{}; 
  
  \node[vertex](d) at(3, 1.1)[label=above right:${y_\gamma=\beta(t_\gamma)}$] {}; 
%  \node[vertex] at(6, 0)[label=above:${y=\beta(t_y)}$] {}; 
  \node[vertex] at (.5, -1) [label=left:$\beta(b)$] {}; 
  \node[vertex] at (9, -1) [label=right:$\beta(a)$] {}; 
 % \node at(3.5, 3)[label=right:$L_\gamma$] {}; 

  \draw[thick, dashed]  (c) to (d){};
  
  \pgfsetlinewidth{1pt}
  \pgfsetplottension{.55}
  \pgfplothandlercurveto
  \pgfplotstreamstart
  \pgfplotstreampoint{\pgfpoint{0.5cm}{-1cm}}
  \pgfplotstreampoint{\pgfpoint{1cm}{.5cm}}
  \pgfplotstreampoint{\pgfpoint{2cm}{0cm}}
  \pgfplotstreampoint{\pgfpoint{3cm}{1.1cm}}
  \pgfplotstreampoint{\pgfpoint{4cm}{-.5cm}}
  \pgfplotstreampoint{\pgfpoint{6cm}{-0cm}}
  \pgfplotstreampoint{\pgfpoint{9cm}{-1cm}}

  \pgfplotstreamend
  \pgfusepath{stroke} 
\end{tikzpicture}
\caption{Surgery III}
\end{figure}

\end{itemize}
\end{lemma}

\subsection{Sublinearly Morse boundary}
We now introduce the definition of $\kappa$-Morse quasi-geodesic, which will be fundamental for our construction. Recall that a function $\kappa:[0,\infty)\rightarrow[1,\infty)$ is \emph{sublinear} if $\lim_{t\rightarrow\infty}\frac{\kappa(t)}{t}=0.$ As shown in \cite[Remark 3.1]{QRT1}, we may assume that $\kappa$ is monotone increasing and concave.
To set the notation, we say a quantity $D$ \emph{is small compared to} a radius $r >0$ if 
\begin{equation} \label{Eq:Small} 
D \leq \frac{r}{2\kappa(r)}. 
\end{equation} 
 
Given a quasi-geodesic ray $\alpha$ and a constant $m$, we define 
$$\mathcal{N}_{\kappa}(\alpha, m) := \Big\{ x \in X \ : \ d(x, \alpha) \leq m \cdot \kappa(\Vert x \Vert) \Big\}.$$

The following observation will be useful.

\begin{definition} \label{D:k-morse}
Let $Z \subseteq X$ be a closed set, and let $\kappa$ be a sublinear function. 
We say that $Z$ is \emph{$\kappa$-Morse} if there exists a proper function 
$m_Z : \mathbb{R}^2 \to \mathbb{R}$ such that for any sublinear function $\kappa'$ 
and for any $r > 0$, there exists $R$ such that for any $(\qq, \sQ)$-quasi-geodesic ray $\beta$
with $m_Z(\qq, \sQ)$ small compared to $r$, if 
$$d_X(\beta_R, Z) \leq \kappa'(R)
\qquad\text{then}\qquad
\beta|_r \subset \calN_\kappa \big(Z, m_Z(\qq, \sQ)\big).$$
The function $m_Z$ will be called a $\emph{Morse gauge}$ of $Z$. 
\end{definition}

Note that we can always assume without loss of generality that $\max \{ \qq, \sQ \} \leq m_Z(\qq, \sQ)$, and we will assume this in the following.

\begin{figure}[H]
\begin{tikzpicture}[scale=0.7]
 \tikzstyle{vertex} =[circle,draw,fill=black,thick, inner sep=0pt,minimum size=.5 mm]
[thick, 
    scale=1,
    vertex/.style={circle,draw,fill=black,thick,
                   inner sep=0pt,minimum size= .5 mm},
                  
      trans/.style={thick,->, shorten >=6pt,shorten <=6pt,>=stealth},
   ]

  \node[vertex] (o) at (0,0)[label=left:$\go$] {}; 
  \node(a) at (11,0)[label=right:$Z$] {}; 
  
  \draw (o)--(a){};
  \draw [dashed] (0, 1.4) to [bend left = 10] (10,3){};
  \draw [dashed] (0, 0.4) to [bend left = 10] (5,2){};
  
   \draw [decorate,decoration={brace,amplitude=5pt},xshift=0pt,yshift=0pt]
  (5, 2) -- (5,0)  node [thick, black,midway,xshift=0pt,yshift=0pt] {};   
  \node at (7,1){\small $m_{Z}(\qq, \sQ) \cdot \kappa(r)$};
     
 \draw [dashed] (10, 4.5) to (10,0) {};
 \node at (10,0)[label=below:$R$] {};
 \draw [dashed] (5, 4.5) to (5,0){};
 \node at (5,0)[label=below:$r$] {};
        
  \draw [decorate,decoration={brace,amplitude=5pt},xshift=0pt,yshift=0pt]
  (10, 3) -- (10,0)  node [thick, black,midway,xshift=0pt,yshift=0pt] {};       
  \node at (11.2, 1.5) { \small $ \kappa'(R)$};
   \node at (10.7, 3) {\small  $ \beta_{R}$};
    \node at (6, 3.5) {\small  $\beta$};
        
  \pgfsetlinewidth{1pt}
  \pgfsetplottension{.75}
  \pgfplothandlercurveto
  \pgfplotstreamstart
  \pgfplotstreampoint{\pgfpoint{0cm}{0cm}}  
  \pgfplotstreampoint{\pgfpoint{1cm}{-.6cm}}   
  \pgfplotstreampoint{\pgfpoint{2cm}{0.4cm}}
  \pgfplotstreampoint{\pgfpoint{3cm}{0cm}}
  \pgfplotstreampoint{\pgfpoint{4cm}{1.5cm}}
  \pgfplotstreampoint{\pgfpoint{5cm}{1cm}}
  \pgfplotstreampoint{\pgfpoint{6cm}{3.2cm}}
  \pgfplotstreampoint{\pgfpoint{7cm}{2cm}}
  \pgfplotstreampoint{\pgfpoint{8cm}{3.8cm}}
  \pgfplotstreampoint{\pgfpoint{9cm}{2.5cm}}
  \pgfplotstreampoint{\pgfpoint{10cm}{3cm}}
  \pgfplotstreamend
  \pgfusepath{stroke} 
\end{tikzpicture}
\caption{Definition of $\kappa$-Morse set $Z$: Every quasi-geodesic ray $\beta$ has the property that there exists $R(Z, r, \qq, \sQ, \kappa')$, such that if $\beta_{R}$ is distance $\kappa'(R)$ from $Z$, then $\beta|_{r}$ is in the neighborhood $\calN_{\kappa}(Z, m_{Z}(\qq, \sQ))$. }
\label{Fig:Strong1} 
\end{figure}
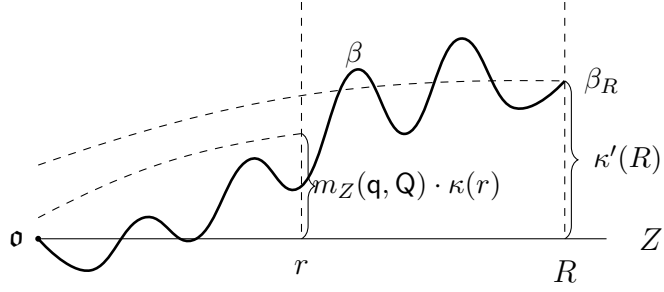

\subsubsection{Boundary definition}

\begin{definition}\label{k-fellow traveling}
Given two quasi-geodesic rays $\alpha$, $\beta$ based at $\go$, we say that $\beta \sim \alpha$ if they sublinearly track each other: 
i.e. if 
$$\lim_{r \to \infty} \frac{d(\alpha_r, \beta_r)}{r} = 0.$$

We denote the equivalence class of $\alpha$ as $\Ga$, and the sublinearly Morse boundary, denoted $\pka X$, is the set of all such equivalence classes.
\end{definition}

By the triangle inequality, $\sim$ is an equivalence relation on the space of quasi-geodesic rays based at $\go$, hence also 
on the space of $\kappa$-Morse quasi-geodesic rays.

\subsubsection{$\kappa$-weakly Morse rays}
As in \cite{QRT1}, we also define a different notion of sublinearly Morse which more closely matches
the usual definition of Morse. 

\begin{definition}\label{defn:weakly-Morse}
Let $Z \subseteq X$ be a closed set, and let $\kappa$ be a concave sublinear function. 
We say $Z$ is \emph{$\kappa$-weakly Morse} if there exists a proper function 
$m_Z : \mathbb{R}^2 \to \mathbb{R}$ such that for any $(\qq, \sQ)$-quasi-geodesic $\gamma \from [s, t] \to X$ with endpoints on $Z$, 
\[
\gamma([s, t]) \subset \calN_\kappa \big(Z, m_Z(\qq, \sQ)\big).
\]
The function $m_Z$ will be called a $\emph{$\kappa$-weakly Morse gauge}$ of $Z$. 
\end{definition}

We note that $\kappa$-weakly Morse and $\kappa$-Morse are equivalent for proper, geodesic metric spaces, see \cite[Remark 3.1]{QRT2}.

%\begin{proposition}\label{Morse-implies-weakly-Morse}
%Let $\alpha: [0, \infty) \to X$ be a quasi-geodesic ray.
%Then $\alpha$ is $\kappa$-Morse if and only if it is $\kappa$-weakly Morse.
%\end{proposition}

% Lastly we present an alternative characterization of $\kappa$-Morse. To proceed we discuss slim triangles. The idea of capturing coarse hyperbolicity by slim triangles first came by Gromov in \cite{gromov}. 

% \begin{definition}

%  Let $b$ be an infinite geodesic ray in a CAT(0) space. We say that b satisfies the $\kappa$-slim condition if there exists some $c \geq 0$ such that for any $x \in X, y \in b$, we have $d(x_b, [x, y]) \leq c\kappa(z)$ for some z between $x_b$ and y. 
% \end{definition}

% It is shown in \cite{MQZ20} That a geodesic ray $\beta_0$ satisfies the $\kappa$-slim condition if and only if $\beta_0$ is $\kappa$-Morse.

\subsubsection{Topology on the sublinearly Morse boundary}

\begin{definition}\label{D:open sets in kappa morse}
    Let $X$ be a proper, geodesic metric space, and let $\kappa$ be a sublinear function. Let $a$ be a (quasi-)geodesic ray representative of $\mathfrak{a}\in\partial_\kappa X$, and let $m_a$ be a Morse gauge for $a$. We define $\mathcal{U}(a,r)$ as follows:
    
    An equivalence class $\Gb\in\partial_\kappa X$ is an element of $\mathcal{U}(a,r)$ if, for any $(q,Q)$-quasi-geodesic $\varphi:[0,\infty)\rightarrow X$ with $\varphi\in\Gb$ and with $m_a(q,Q)$ small compared to $r$, we have
    $$\varphi([0,t_r])\subseteq\mathcal{N}_\kappa(a,m_a(q,Q)).$$
\end{definition}

For a proper, geodesic metric space $X$, the collection of sets $\mathcal{U}(a,r)$ form a neighborhood basis for $\Ga$, and in particular $\Ga\in\mathcal{U}(a,r)$ \cite[Lemma 4.2]{QRT2}.

%\begin{notation}
%    We will denote the gothic letters $\Ga, \Gb \ldots$ to be the elements of $\p_\K X$ without referring to an element in the equivalence class. When $\Ga, \Gb \ldots$ are defined, the normal letters $a, b, \ldots$ will be geodesic rays in the equivalence classes of $\Ga, \Gb, \ldots$, respectively, that emanate from the chosen basepoint $\ob$. All quasi-geodesics are denoted using Greek letters like $\gamma, \lambda, \eta, \xi$. Given a quasi-geodesic ray $\gamma: [0,\infty) \rightarrow X$, if $x_1,x_2$ are points on $\gamma$, we define $[x_1,x_2]_\gamma$ to be the segment of $\gamma$ connecting $x_1$ to $x_2$. If there is no subscript on a segment, like $[x_1,x_2]$, this will denote a geodesic connecting $x_1$ to $x_2$. As there is not necessarily a unique geodesic between two points, we will specify the geodesic when needed. For a quasi-geodesic $\gamma$ and $r>0$, let $t_r$ be the first time where $||\gamma(t)|| = r$. We define $$\gamma_r := \gamma(t_r) \quad \text{and} \quad \gamma|_r := \gamma[0, t_r] = [\gamma(0), \gamma_r]_\gamma.$$
%\end{notation}

\subsection{Quasi-redirecting boundary}
Here we collect all facts following the set-up of \cite{QR24}. Please refer to \cite{QR24} for a complete treatment.
\subsubsection{Equivalence classes of rays up to quasi-redirection} 
Recall that, for quasi-geodesic rays $\alpha$ and $\beta$, we say that $\alpha \preceq \beta$ if $\alpha$ can be \emph{quasi-redirected} to $\beta$, that is, if there is a family of quasi-geodesic 
rays with uniform constants that coincide with $\alpha$ in the beginning for an arbitrarily long time 
but eventually coincide with $\beta$. We now formalize this definition. 

\begin{definition}\label{Def:Redirection}
Let $\alpha, \beta$ and $\gamma$ be quasi-geodesic rays. We say $\beta$ \emph{eventually coincide with 
$\gamma$} (and write $\gamma \sim \beta$) if there are times $t_\beta, t_\gamma >0$ such that, 
for $t\geq t_\gamma$, we have 
\[
\gamma(t) =\beta(t+t_\beta).
\]
For $r>0$, we say $\gamma$ \emph{quasi-redirects $\alpha$ to $\beta$ at radius $r$} if
\[
\gamma|_r = \alpha|_r \qquad\text{and} \qquad \beta \sim \gamma. 
\]
If $\gamma$ is a $\qq$--ray, we say \emph{$\alpha$ can be $\qq$--redirected to $\beta$
at radius $r$}. We refer to $t_\gamma$ as the \emph{landing time}. 	
We say $\alpha \preceq \beta$, if there is $\qq \in [1, \infty) \times[0,\infty)$ such that, 
for every $r >0$, $\alpha$ can be $\qq$--redirected to $\beta$ at radius $r$. \end{definition}

\begin{figure}[H]
\centering
\begin{tikzpicture}[scale=0.7] 
\draw[blue, thick] (0,0) -- (9,0) node[below] {$\beta$};
\draw[red, thick] (0,0) -- (0,4) node[left] {$\alpha$};
\draw[|-|, thin] (-0.4,0)--(-0.4,1.8);
\node[left] at (-0.4,0.9) {$r$};
\node at (2,2) {$\gamma$};

 \pgfsetlinewidth{1pt}
 \pgfsetplottension{.75}
 \pgfplothandlercurveto
 \pgfplotstreamstart
 \pgfplotstreampoint{\pgfpoint{0 cm}{0cm}}  
 \pgfplotstreampoint{\pgfpoint{.3 cm}{3 cm}}   
 \pgfplotstreampoint{\pgfpoint{2 cm}{1.4 cm}}
 \pgfplotstreampoint{\pgfpoint{3 cm}{.7 cm}}
 \pgfplotstreampoint{\pgfpoint{4.5 cm}{.3 cm}}
 \pgfplotstreampoint{\pgfpoint{6 cm}{.1 cm}} 
 \pgfplotstreampoint{\pgfpoint{8 cm}{.005 cm}} 
 \pgfplotstreampoint{\pgfpoint{9 cm}{0 cm}} 
 \pgfplotstreamend
 \pgfusepath{stroke} 
 \end{tikzpicture}
\caption{The ray $\alpha$ can be quasi-redirected to $\beta$ at radius $r$.}
\end{figure}
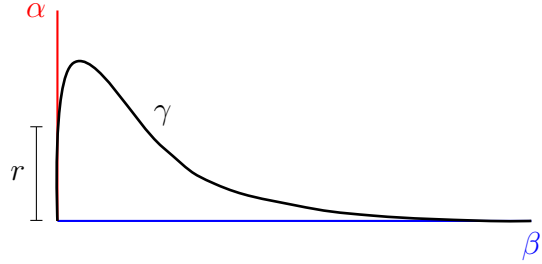

\begin{definition}
Define $\alpha \simeq \beta$ if and only if $\alpha  \preceq \beta$ and $\beta  \preceq \alpha$. Then  
$\simeq$ is an equivalence relation on the space of all quasi-geodesic rays in $X$. Let $P(X)$
denote the set of all equivalence classes of quasi-geodesic rays under $\simeq$. 
For a quasi-geodesic ray $\alpha$, let $[\alpha] \in P(X)$ denote the equivalence class containing 
$\alpha$. We extend $\preceq$ to $P(X)$ by defining $[\alpha] \preceq [\beta]$ if 
$\alpha \preceq \beta$. Note that this does not depend on the representative chosen 
in the given class. The relation $\preceq$ is a partial order on elements of $P(X)$.
\end{definition}

\begin{proposition}\cite[Lemma 2.5, Lemma 3.2, Proposition 3.4]{QR24}
\begin{itemize}
\item Let $G$ be a finitely generated group, then $Cay(G)$ satisfies Assumption 0.
\item Quasi-redirecting is transitive: 
\[ \alpha \preceq \beta \text{ and } \beta \preceq \gamma \Longrightarrow \alpha\preceq \gamma. \] In particular, Let $\alpha, \beta, \gamma$ be  quasi-geodesic rays. If $\alpha$ can be $(q_1, Q_1)$--redirected to 
$\beta$ at every radius $r>0$ and $\beta$ can be $(q_2, Q_2)$--redirected to $\gamma$ at every
radius $r>0$, then $\alpha$ can be $(q_3, Q_3)$--redirects to $\gamma$ at every radius $r>0$ where 
\[
q_3 = \max \big\{ q_2+ 1, q_1 \big\}, \, \text{ and }  \, Q_3= \max \big\{ Q_1, Q_2 \big\}.
\]
 
\item Quasi-isometry of $X$ induces a an automorphism on $P(X)$.

\end{itemize}
\end{proposition}

There are two further technical assumptions made to rule out the wild spaces and to ensure that there exists a sensible topology. 

\subsection*{Assumption 1} (Quasi-geodesic representative)\label{Q1}
There is $\qq_0$ (by making it larger, we can assume it is the same at $\qq_0$ in 
Assumption~0) such that every equivalence class $\bfa \in P(X)$ contains a 
$\qq_0$--ray. We fix such a $\qq_0$--ray, denote it by $\alpha_0 \in \bfa$ and refer to it 
as the central element of the class $\bfa$.

\subsection*{Assumption 2} (Uniform redirecting function)\label{Q2}
For every $\bfa \in P(X)$, there is a function 
\[
f_\bfa \from [1, \infty) \times [0, \infty) \to [1, \infty) \times [0, \infty),
\] 
called the redirecting function of
the class $\bfa$, such that if $\bfb \prec \bfa$ then any $\qq$--ray $\beta \in \bfb$ can be
$f_\bfa(\qq)$--redirected to $\alpha_0$. 

Note that the function $f_\bfa$ may depend on the choice of the central element. But such 
functions exist for every quasi-geodesic ray, as we show in the following:

The topology on $X \cup P(X)$ is defined via a system of neighbourhoods. Recall that points
in $P(X)$ are equivalence classes of quasi-geodesic rays. To unify the treatment of point in 
$X$ and $P(X)$, for every $x \in X$, we consider the set of quasi-geodesic rays that pass 
through $x$. Abusing the notation, we denote this set again by $x$, that is 
\[
x = \Big\{ \text{quasi-geodesics rays passing through $x$} \Big\}. 
\]
We use the gothic letters $\Ga, \Gb, \Gc$ to denote elements of $P(X) \cup X$, that is, either a set of quasi-geodesic rays 
passing through a point $x \in X$ or an equivalence class of quasi-geodesic rays in $P(X)$. For $\bfa \in P(X)$, define
$F_\bfa \from [1,\infty) \times [0, \infty) \to [1,\infty) \times [0, \infty)$ by 
\[
F_\bfa(\qq) = {\bold f}_\bfa(\qq) + (1,q), \qquad\text{for}\qquad \qq \in   [1,\infty) \times [0, \infty).
\]

\begin{definition} \label{Def:The-In-Topology}
For $\bfa \in P(X)$ and $r>0$, define   
\[
 \calU(\bfa, r) := \Big  \{  \Gb \in P(X) \cup X \, \ST 
 \text{ every $\qq$--ray in $\Gb$ can be $F_\bfa(\qq)$--redirected to $a_0$ at $r$} \Big \}.   
\]
\end{definition}

In most arguments about a class of quasi-geodesic rays $\bfa$, it is enough to consider $\qq$--rays 
where $\qq$ is not too big. We now make this precise. For $\qq=(q, Q)$ and $\qq'=(q', Q')$ we
say $\qq \leq \qq'$ if $q \leq q'$ and $Q \leq Q'$. 

\begin{lemma} \label{Lem:Max}
For every $r>0$ there is a pair of constants $\qq_{\rm max} \in [1,\infty) \times [0,\infty)$
such that if $\qq \not \leq \qq_{\rm max}$ then, for every $\bfa \in P(X)$, any $\qq$--ray $\beta$ can be 
$F_\bfa(\qq)$--redirected to $a_0$ at radius $r$. 
\end{lemma}

\subsection*{A system of neighbourhoods}
For each $\bfa \in P(X)$, define 
\[
\calB(\bfa) = \Big\{ \calV \subset X \cup P(X) \ST 
\calU(\bfa, r) \subset \calV \quad \text{ for some $r>0$ } \Big\}
\]
and for every $x \in X$, define 
\[
\calB(x) = \Big\{ \calV \subset X \cup P(X) \ST 
B(x, r) \subset \calV \quad \text{ for some $r>0$ } \Big\}.
\]

\begin{theorem}\cite[Theorem 5.9]{QR24}
The space $X \cup \partial X$ is a bordification of the space $X$ and $\partial X$ is QI-invariant.
\end{theorem} 

\section{Dynamics on sublinearly Morse boundaries}

In this section, we prove Theorem \ref{thm:minimality of sublinear}, i.e., the $G$ orbit of every element $\Ga\in\pka G$ is dense in $\pka G$. We do this by showing a slightly weaker statement: the $G$ action on $\pka X$ is minimal as long as $G\curvearrowright X$ is cobounded.  In this section, we will assume that $X$ satisfies assumption 0.

\subsection{Minimality of the sublinearly Morse boundary}

We begin by proving a fact on $\partial X$. This fact is very straightforward: for any $\Gb\in\partial X$, there exists $\Ga\in\partial X$ so that $\Ga$ can be translated away from $\go$ along $\Gb$, and vice versa. Let $b\in\Gb\in\partial X$. Recall that  a sequence of group elements $\{g_i\}$ \emph{tracks} $b$ if, for every $T\geq0$, there exists $M\geq0$ so that, for every $i\geq M$, there exists $t\geq T$ with $d(g_i\go,\beta(t))\leq K$.

\begin{lemma}\label{Lem:g_ia leaves ball of radius R} Let $K\geq 0$ be the cobounded constant of $G\curvearrowright X$ and assume $|\partial X| \geq 3$. For any $\Gb \in \partial X $, choose a geodesic ray $b \in \Gb$ and let $\{g_i\}$ be any sequence in $G$ that tracks $b$. 
Then, there exists $\Ga \in \partial X $ such that for any $a \in \Ga $ and $R>0$, there exists $j \in \mathbb{N}$ such that $g_i a\cap B_R(\ob) = \emptyset$ for all $i\geq j$. In addition, for any sequence $\{h_i\}$ that tracks $a$, and for any $R>0$ there exists $j\in\mathbb{N}$ such that $h_ib\cap B_R(\ob)=\emptyset$ for all $i\geq j$. 
    
\end{lemma}

\begin{proof}
Since $(g_i)_i$ tracks $\beta$, $d(g_i\cdot \ob) \rightarrow \infty$ as $i \rightarrow \infty$. Let $\Ga,\Gc \in \partial_\K X$ so that $\Ga\neq \Gc \neq \Gb$. Let $a\in \Ga$ and $c \in \Gc$ be geodesic representatives. For each $i$, let $p_i \in \pi_{g_i a}(\ob)$ and $q_i \in \pi_{g_i c}(\ob)$. For the sake of contradiction, assume that both sequences $\{||p_i||\}$ and $\{||q_i||\}$ have a subsequence bounded above by some $R > 0$. 
By passing to a subsequence, we may assume both sequences $\{||p_i||\}$ and $\{||q_i||\}$ are bounded above by $R$. Note that this implies $d(g_i \ob, p_i) \rightarrow \infty$ and $d(g_i \ob, q_i) \rightarrow \infty $, so we get that $\{g_i^{-1} p_i\}$ and $\{g_i^{-1} q_i\}$ are unbounded sequences. For each $i$, we have $d(p_i, q_i) < 2R$. Thus,  $d(g_i^{-1} p_i, g_i^{-1} q_i) < 2R$. This gives two unbounded sequences $\{g_i^{-1} p_i\}$ and $\{g_i^{-1} q_i\}$ such that $d(g_i^{-1} p_i, g_i^{-1} q_i) < 2R$. This implies $a$ and $c$ fellow travel which gives $a \simeq c$, a contradiction to $\Ga\neq \Gc$. 

Without loss of generality we may assume $\{||p_i||\}$ is unbounded. Now assume, for contradiction, that there exists an infinite sequence $\{x_i\}$ with $x_i\in h_ib\cap B_R(\ob)$ for some $R>0$. Let $k_i\go\in B_K(x_i)$. Then $\{k_i\}$ is a sequence in $G$ that tracks $b$, but $d(k_i\go,\go)\leq K+R$ for all $i$, a contradiction. 
\end{proof}

\begin{remark}\label{remark: statement on QR applies to kappa boundary}
    As discussed in \cite[Section 6]{QR24}, sublinearly Morse elements are minimal in $\partial X$. In particular, one may replace $\partial X$ with $\pka X$ in Lemma \ref{Lem:g_ia leaves ball of radius R}.
\end{remark}

We now introduce a lemma which, given a geodesic representative $a$ of $\Ga$ and  group element $g$, constructs a uniform quality quasi-geodesic from $ga$ to $b$. As in Remark \ref{remark: statement on QR applies to kappa boundary}, the following lemma equally applies to $\pka X$.

\begin{lemma}\label{Lem:(81,2K)-quasi-geo-ray}
     Let $G$ be a group that acts cocompactly on $X$ with cocompact constant $K$ and $|\partial X | \geq 3$. For any $\Gb \in \partial X $, choose a geodesic $b \in \Gb$ and sequence $\{ g_i\}$ that tracks $b$. Let $a\in \Ga$ be as in Lemma \ref{Lem:g_ia leaves ball of radius R}. For any $p_i \in \pi_{g_i\cdot a}(\ob)$, there exists a $(27,3K)$-quasi-geodesic ray that contains $[\ob, p_i]$ whose tail end is $b$. 
\end{lemma}

\begin{proof}
    There exists an $r\geq 0$ such that $d(g_i\cdot \ob, b_r) < K$ by definition of the cocompact action. Thus, by Lemma \ref{surgery} $\gamma_i = [\ob, p_i]*[p_i, g_i\cdot \ob]_{g_i \cdot a} * [g_i\cdot \ob, b_r]$ is a $(3,K)$-quasi-geodesic. Let $R>0$ be such that $B_\ob(R)$ contains $\gamma_i$. Denote $q_i = \pi_{\gamma_i}(b_R)$. We have,

    \begin{align*}
        ||q_i|| = d(\ob,q_i) &\geq d(\ob, b_R) - d(b_R, q_i)\\
                    &\geq R - d(b_r, b_R)\\
                    &= R-(R-r)\\
                    &= r.
    \end{align*}        
    Also, we see that $d(\ob, p_i) \leq r+K$ because $d(\ob, g_i\cdot \ob) \leq r+K$ and $p_i$ is a closest point projection. We now break into cases,

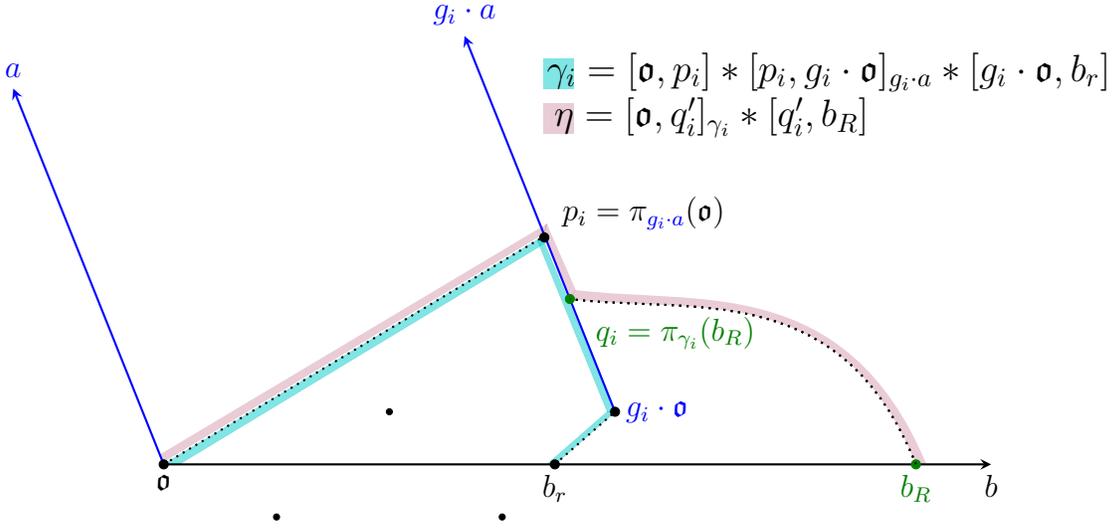
\begin{figure}[h]
    \centering
            \begin{tikzpicture}[scale=1]

%   help lines 
 % \draw[step=1cm,gray, opacity=.5] (1,-2) grid (15,10);
\definecolor{aqua}{rgb}{0.0, 1.0, 1.0}

% \gamma_i
    \draw[black!25!aqua, fill = black!20!aqua, opacity=0.5  ] plot((1.2,-2) -- (1,-2) -- (6.06,1.04) -- (6.11,.97) -- cycle;
    \draw[black!25!aqua, fill = black!20!aqua, opacity=0.5  ] plot((6.06-.03,1.04-.03) -- (6.03-.06,1.01-.03) -- (6.9,-1.3)  -- (7,-1.3) -- cycle;
    % \fill[black!25!aqua, fill = black!20!aqua, opacity=0.5 ] (6.15,-2) .. controls (6.25, -1.3) .. (6.75,-.8) -- (6.85,-.83) .. controls (6.35, -1.33) .. (6.25,-2.03) -- cycle;
    \draw[black!25!aqua, fill = black!20!aqua, opacity=0.5  ] plot((6.9,-1.3)  -- (7,-1.3) -- (6.2,-2)  -- (6.1,-2) -- cycle;

% \eta
    \draw[black!25!purple, fill = black!20!purple, opacity=0.2  ] plot((.96,-2+.13) -- (1,-2+.04) -- (6.06,1.04+.04) -- (6.14-.05,1.17+.02) -- cycle;
    \draw[black!25!purple, fill = black!20!purple, opacity=0.2  ] plot((6.06,1.04+.04) -- (6.14 -.05,1.17+.05-.03) -- (6.5-.03,0.35-.03) -- (6.4,0.2) -- cycle;

    \fill[black!25!purple, fill = black!20!purple, opacity=0.2  ] (11,-2) .. controls (10, .5) and (8,0) .. (6.4,0.2) -- (6.5-.03,0.35-.03) .. controls (8.1-.03, .15-.03) and (10.1-.03,.6-.035) .. (11.17-.03, -2) -- cycle;

\draw [thick, dotted](11,-2) .. controls (10, .5) and (8,0) .. (6.4,0.2);

% Legend
    \fill[black!25!aqua, fill = black!20!aqua, opacity=0.5] (6.75-.7,3) rectangle (7.15-.7,3.4);
     \node[right,  font=\large] at (6-.05, 3.22) {$ \gamma_i = [\ob, p_i]*[p_i, g_i\cdot \ob]_{g_i \cdot a} * [g_i\cdot \ob, b_r]$};

     \fill[black!25!purple, fill = black!20!purple, opacity=0.2 ] (6.75-.7,3-.6) rectangle (7.15-.7,3.4-.6);
     \node[right,  font=\large] at (6+.05, 3.22-.6) {$ \eta = [\ob, q_i']_{\gamma_i}*[q_i', b_R]$};

% geodesic a
    \draw [blue,thick, ->](1,-2) to (-1,3);
    \node[blue, above] at (-1,3) {$a$};

%geodesic g_i \cdot a
    \draw [blue,thick, ->](7,-1.3) to (5,3.7);
    \node[blue, above] at (5,3.7) {$g_i\cdot a$};
    \node[blue, right] at (7,-1.3) {$g_i\cdot \ob$};

% geodesic b
    \draw [thick, ->](1,-2) to (12,-2);
    \node [below] at (12,-2) {$b$};
    \draw[fill] (1,-2) circle [radius=0.06];
    \node[below] at (1,-2) {$\ob$};

% orbit points of the \ob
     \draw[fill] (2.5,-2.7) circle [radius=0.04];
     \draw[fill] (4,-1.3) circle [radius=0.04];
     \draw[fill] (5.5,-2.7) circle [radius=0.04];
     \draw[fill] (7,-1.3) circle [radius=0.06];

% Closest point projection of \ob 
    \draw [thick, dotted](1,-2) to (6,1);
    \draw[fill] (6.06,1.02) circle [radius=0.06];
    \node [right] at (6.16,1.34) {$p_i=\pi_{\color{blue}g_i\cdot a}(\ob)$};

%  b_r
    \draw[fill] (6.2,-2) circle [radius=0.06];
    \node[below] at (6.2,-2) {$b_r$};

    % \draw[fill, red] (6.8,-.8) circle [radius=0.06];
    % \node[red, right] at (6.8,-.5) {$q_i = \pi_{\color{blue}g_i\cdot \ob}(b_r)$};

    \draw [thick, dotted](6.2,-2) to (7,-1.3);

% b_R
    \draw[fill, black!50!green] (11,-2) circle [radius=0.06];
    \node[below, black!50!green] at (11,-2) {$b_R$};

    \draw[fill, black!50!green] (6.4,0.2) circle [radius=0.06];
    \node[right, black!50!green] at (6.6,-.25) {$ q_i = \pi_{\gamma_i}(b_R) $};

    \draw [thick, dotted](11,-2) .. controls (10, .5) and (8,0) .. (6.4,0.2);

    \end{tikzpicture}
    \caption{A picture of CASE 1. We have $\eta$ will be a (9,K) quasi-geodesic that contains the geodesic segment $[\ob, p_i]$.}
    \label{fig:Case 1 of building quasi-geo}
\end{figure}
    
    CASE 1: Suppose $q_i \notin [\ob,p_i]$, then choosing $\eta = [\ob, q_i]_{\gamma_i}*[q_i, b_R]$ is a $(9,K)$-quasi-geodesic. Furthermore, as $||q_i|| \leq R = ||b_R||$, we get that for any $x \in b[R,\infty)$, $\pi_\eta(x) = b_R$, via an argument from \cite[Lemma 4.3]{QRT1}. Hence $\eta *b[R,\infty)$ is an $(27,K)$-quasi geodesic that fellow travels $b$. See Figure \ref{fig:Case 1 of building quasi-geo}.

    CASE 2: In the case that $q_i \in [\ob,p_i]_{\gamma_i}$, then $q_i$ is within $K$ of $p_i$. Indeed, since  $||q_i|| \geq r$ and $||p_i|| \leq r + K$, the fact that both $q_i$ and $p_i$ are on the geodesic $[\ob,p_i]_{\gamma_i}$ emanating from $\ob$ implies $d(q_i,p_i)\leq K$. Thus, $$\eta' = [\ob, q_i]_{\gamma_i}*[q_i, p_i]_{\gamma_i}*[p_i, q_i]*[q_i, b_R]$$ is a $(9,3K)$-quasi-geodesic. Similar to CASE 1, we find a $(27,3K)$-quasi-geodesic that fellow travels $b$.
\end{proof}

\begin{corollary}\label{cor:Morse Gauge Control}
    Given the conditions of Lemma \ref{Lem:(81,2K)-quasi-geo-ray}, if $b$ is $\kappa$-Morse, then the $(27,3K)$-quasi-geodesic ray found in Lemma \ref{Lem:(81,2K)-quasi-geo-ray} is $\kappa$-Morse with Morse gauge depending only on $K$ and the Morse gauge of $b$.
\end{corollary}

\begin{proof}
    This is immediate from Lemma \ref{Lem:(81,2K)-quasi-geo-ray} and  \cite[Corollary 3.5]{QRT2}.
\end{proof}

\begin{remark}\label{Remark:Symmetry in construction}
    Notice that Lemma \ref{Lem:g_ia leaves ball of radius R} gives symmetric results: translating the basepoint of $a$ along $b$ leaves every ball of radius $R$, and vice-versa. Therefore, by just changing letters in the proof of Lemma \ref{Lem:(81,2K)-quasi-geo-ray}, we can prove that there exists a $(27,3K)$-quasi-geodesic which first projects to an orbit of $b$, then eventually fellow travels $a$. This observation will be important point in proving Theorem \ref{density}.
\end{remark}

To summarize the above lemma, we have found a quasi-geodesic which first nearest-point projects to $g_i a$ and then, eventually, fellow travels $b$. In this next lemma, we find a quasi-geodesic $\lambda_i$ which closest point projects to $g_i a$ and then fellow travels $g_i a$. Corollary \ref{cor:Morse Gauge Control} then gives us control over the Morse gauge for $\lambda_i$ in terms of only $K$ and the Morse gauge of $b$.

\begin{lemma}\label{Lem: Lambda is k-Morse}
      Let $G$ be a group that acts cocompactly on $X$ with cocompact constant $K$ and $|\partial_\K X | \geq 3$. For any $\Gb \in \partial_\K X $, choose a geodesic $b \in \Gb$ and sequence $\{ g_i\}$ that tracks $b$. Let $a\in \Ga$ be as in Lemma \ref{Lem:g_ia leaves ball of radius R}. For any $p_i \in \pi_{g_i\cdot a}(\ob)$, we have $\lambda_i = [\ob,p_i]*[p_i, g_i\cdot a(\infty)]$ is a $(3,0)$-quasi-geodesic that is $\K$-Morse with Morse gauge depending only on $a,b$ and $K$.
\end{lemma}
 
\begin{proof}
    See Figure \ref{fig:lambda_i is K-Morse}. Consider any $\xi \in [\lambda_i]$. Let $q_i \in \pi_\xi(p_i)$. Then, $[\ob,q_i]_\xi*[q_i, p_i]$ is a $(3\qq,\sQ)$-quasi-geodesic with endpoints on $[\ob, p_i]$ which is contained in the $\eta$ found in Lemma \ref{Lem:(81,2K)-quasi-geo-ray}. By the weakly $\K$-Morse condition, $$[\ob,q_i]_\xi*[q_i, p_i] \subset \calN_\kappa \big(\eta, m_\eta(3\qq, \sQ)\big),$$ and by Corollary \ref{cor:Morse Gauge Control}, $m_\eta$ depends only on $b$ and $K$. 
    Similarly, $[\ob,q_i]*[q_i,\xi(\infty)]_{\xi}$ is a $(3\qq,\sQ)$-quasi-geodesic that fellow travels $g_i\cdot a$. Thus $$[\ob,q_i]*[q_i, \xi(\infty)]_{\xi} \subset \calN_\kappa \big(a, m_a(3\qq, \sQ)\big).$$ 
    Hence, we conclude $$\xi \subset \calN_\kappa \big(g_i\cdot a, m_\eta(3\qq, \sQ) + m_a(3\qq,\sQ) \big).$$
\end{proof}

\begin{figure}[h]
    \centering
           \begin{tikzpicture}[scale=1]

%   help lines 
 % \draw[step=1cm,gray, opacity=.5] (1,-2) grid (15,10);
\definecolor{aqua}{rgb}{0.0, 1.0, 1.0}

% \lambda_i
    \draw[black!25!purple, fill = black!20!purple, opacity=0.2  ] plot((.96,-2+.13) -- (1+.1,-2+.04) -- (9.27,.52) -- (9.27-.1,.52+.1) -- cycle;
    \draw[black!25!purple, fill = black!20!purple, opacity=0.2  ] plot((9.27,.52) -- (-1+8+1,3+.7) -- (-1+8+1-.1,3+.7-.05) -- (9.27-.1,.52-.05) -- cycle;

% Legend

     \fill[black!25!purple, fill = black!20!purple, opacity=0.2 ] (7-.2-6-.8,3-.6) rectangle (7.4-.2-6+.15-.8,3.4-.6+.1);
     \node[right,  font=\large] at (6.75+.05-6.07-.8, 3.22-.63) {$ \lambda_i = [\ob, p_i]*[p_i, g_i\cdot a(\infty)]$};

% geodesic a
    \draw [blue,thick, ->](1,-2) to (-1,3);
    \node[blue, above] at (-1,3) {$a$};

%geodesic g_i \cdot a
    \draw [blue,thick, ->](1+8+1,-2+.7) to (-1+8+1,3+.7);
    \node[blue, above] at (-1+8+1,3+.7) {$g_i\cdot a$};
    \node[blue, right] at (1+8+1,-2+.7) {$g_i\cdot \ob$};

% geodesic b
    \draw [thick, ->](1,-2) to (12,-2);
    \node [below] at (12,-2) {$b$};
    \draw[fill] (1,-2) circle [radius=0.06];
    \node[below] at (1,-2) {$\ob$};

% orbit points of the \ob
     \draw[fill] (2.5,-2.7) circle [radius=0.04];
     \draw[fill] (4,-1.3) circle [radius=0.04];
     \draw[fill] (5.5,-2.7) circle [radius=0.04];
     \draw[fill] (7,-1.3) circle [radius=0.04];
     \draw[fill] (8.5,-2.7) circle [radius=0.04];
     \draw[fill] (10,-1.3) circle [radius=0.06];

% Closest point projection of \ob 
    \draw [thick, dotted](1,-2) to (9.27,.52);
    \draw[fill] (9.27,.52) circle [radius=0.06];
    \node [right] at (9.37,.62) {$p_i=\pi_{\color{blue}g_i\cdot a}(\ob)$};

% quasi geodesic \xi

 \pgfsetlinewidth{1pt}
  \pgfsetplottension{.75}
  \pgfplothandlercurveto
  \pgfplotstreamstart
  \pgfplotstreampoint{\pgfpoint{1cm}{-2cm}}  
  \pgfplotstreampoint{\pgfpoint{2cm}{0cm}}   
  \pgfplotstreampoint{\pgfpoint{3cm}{-.5cm}}
  \pgfplotstreampoint{\pgfpoint{4cm}{.5cm}}
  \pgfplotstreampoint{\pgfpoint{5cm}{0cm}}
  \pgfplotstreampoint{\pgfpoint{6cm}{1.5cm}}
  \pgfplotstreampoint{\pgfpoint{7.5cm}{1.5cm}}
  \pgfplotstreampoint{\pgfpoint{6.3cm}{3cm}}
  \pgfplotstreampoint{\pgfpoint{7cm}{3.2cm}}
  \pgfplotstreampoint{\pgfpoint{6.7cm}{4cm}}
  \pgfplotstreamend
  \pgfsetarrowsend{>}
  \pgfusepath{stroke} 
  \pgfsetarrowsend{}

  \node [above] at (6.7,4) {$\xi$};

% Projection of p_i to \xi

\draw [thick, dotted](7.5,1.5) to (9.27,.52);
\draw[fill] (7.5,1.5) circle [radius=0.06];
\node [below] at (7.5-.6,1.5-.1) {$\pi_{\xi}(p_i) = q_i$};

    \end{tikzpicture}
    \caption{A figure of Lemma \ref{Lem: Lambda is k-Morse}. We subdivide $\xi$ into two parts. The initial segment can be leveraged by using the $\K$-Morseness of $\eta$ in Lemma \ref{Lem:(81,2K)-quasi-geo-ray}. The remaining part of $\xi$ fellow travels $g_i \cdot a$, so we leverage the $\K$-Morseness of $g_i\cdot a$.}
    \label{fig:lambda_i is K-Morse}
\end{figure}
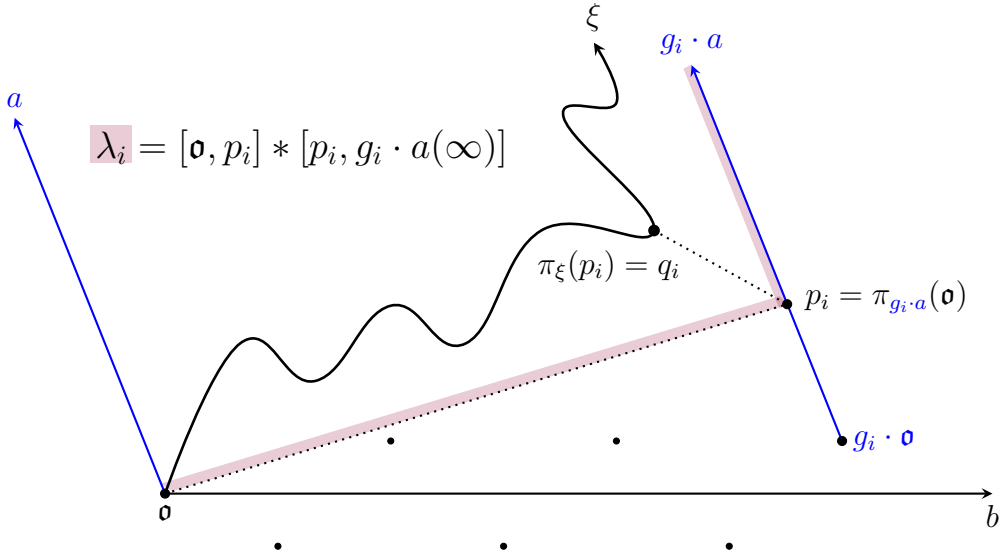

\begin{remark}\label{Rem: Not-dependent on i}
    Notably, the Morse gauge for $\lambda_i$ does not depend on $i$.
\end{remark}

Notice that the construction of each $\lambda_i$ begins by projecting to a point on $g_i a$. Since we are choosing the sequence $g_i$ so that $g_i$ stays close to $b$ and gets farther and farther from $\go$, it is not surprising for us to find that the $\lambda_i$ end up staying sublinearly close to $\beta$, as we show in the next proposition.

\begin{proposition}\label{Prop: Lambda itself is K-close}
    Let $G$ be a group that acts cocompactly on $X$ with cocompact constant $K$ and $|\partial_\K X | \geq 3$. For any $\Gb \in \partial_\K X $, choose a geodesic $b \in \Gb$ and sequence $\{ g_i\}$ that tracks $b$. Let $a\in \Ga$ be as in Lemma \ref{Lem:g_ia leaves ball of radius R}. For  $p_i \in \pi_{g_i\cdot a}(\ob)$, denote $\lambda_i = [\ob,p_i]*[p_i, g_i\cdot a(\infty)]$. For any $r>0$ there exists an $i$ such that $$\lambda_i|_r \subset  \calN_\kappa \big(b, m_b(9, 0) \big).$$
\end{proposition}

\begin{proof}
    Set $\K' = m_b(9,0)\K$. Since $b$ is $\K$-Morse, there exists an $R = R(3,0,r,\K')$ such that the conditions of the $\K$-Morse property holds. By Lemma \ref{Lem:g_ia leaves ball of radius R}, There exists an $i$ such that for $\lambda_i = [\ob,p_i]*[p_i, g_i\cdot a(\infty)]$, we have $||p_i|| \geq R$. By Lemma \ref{Lem:(81,2K)-quasi-geo-ray}, $d(\lambda_i(R), b) \leq m_b(9,0)\K(\lambda_i(R)) = m_b(9,0)\K(R).$ Hence, as $b$ is $\K$-Morse, $$\lambda_i|_r \subset  \calN_\kappa \big(b, m_b(9, 0) \big). $$
\end{proof}

Notice that, referring to Definition \ref{D:open sets in kappa morse}, we have just shown that for every $r>0$, there exists $i$ so that $\lambda_i\in\mathcal{U}(\beta,r)$. However, in order to satisfy the full conditions of Definition \ref{D:open sets in kappa morse}, we need to show that the entire equivalence class of $\lambda_i$ is contained in $\mathcal{U}(\beta,r)$. 
This fact is straightforward: If $\xi$ is in the same equivalence class as $\lambda_i$, then $\xi$ and $\lambda_i$ sublinearly fellow travel in the sense of Definition \ref{k-fellow traveling}. Since $\lambda_i$ sublinearly follows $\beta$ up to distance $r$, and $\xi$ sublinearly follows $\lambda_i$ for all time, $\xi$ must also sublinearly travel $\beta$ up to some distance $r'$. We formalize this argument in the next proposition.

\begin{proposition}\label{Prop: lambda_i is in the open set}
    Let $G$ be a group that acts cocompactly on $X$ with cocompact constant $K$ and $|\partial_\K X | \geq 3$. For any $\Gb \in \partial_\K X $, choose a geodesic $b \in \Gb$ and sequence $\{ g_i\}$ that tracks $b$. Let $a\in \Ga$ be as in Lemma \ref{Lem:g_ia leaves ball of radius R}. For  $p_i \in \pi_{g_i\cdot a}(\ob)$, denote $\lambda_i = [\ob,p_i]*[p_i, g_i\cdot a(\infty)]$. For any $r>0$, there exists an $i$ such that any $(\qq,\sQ)$-quasi-geodesic $\xi$ such that $\xi \sim \lambda_i$ has $\xi|_r \subset \calN_\kappa \big(b, m_b(\qq, \sQ) \big). $
\end{proposition}

\begin{proof}
    Choose $R>0$ to be sufficiently large and $i$ such that \[\lambda_i|_R \subset  \calN_\kappa \big(b, m_b(9, 0) \big). \] Specifically, we can choose $R$ to be larger than the $R(3,0,r,\kappa')$ in Proposition \ref{Prop: Lambda itself is K-close}  and also larger than $2r$. Pick any $(\qq, \sQ)$-quasi-geodesic $\xi$ such that $\xi \sim \lambda_i$. By being in the same equivalence class, \[\xi|_r \subseteq \calN_\kappa \big(\lambda_i, m_{\lambda_i}(\qq,\sQ) \big).\] 
    
    Since $m_{\lambda_i}$ is independent of $i$ by Remark~\ref{Rem: Not-dependent on i}, we denote $m_{\lambda_i}$ by $m_{\lambda}$. For any $x \in [\ob, \xi_r]_\xi$, by \cite[Lemma 2.2]{QRT2},
    \[||\pi_{\lambda_i}(x)|| \leq 2 ||x|| \leq 2r .\]
    
    Since $R>2r$, $$d\bigg(\pi_b\big(\pi_{\lambda_i}(x)\big), \pi_{\lambda_i}(x)\bigg) \leq m_b(9,0)\K(\pi_{\lambda_i}(x)) \leq 2m_b(9,0)\K(x). $$ That is, for any $r>0$, we can find an $i$ such that all $\xi \sim \lambda_i$ have $$\xi|_r \subset \calN_\kappa \big(b, 2m_b(9, 0)+m_\lambda(\qq, \sQ) \big). $$

    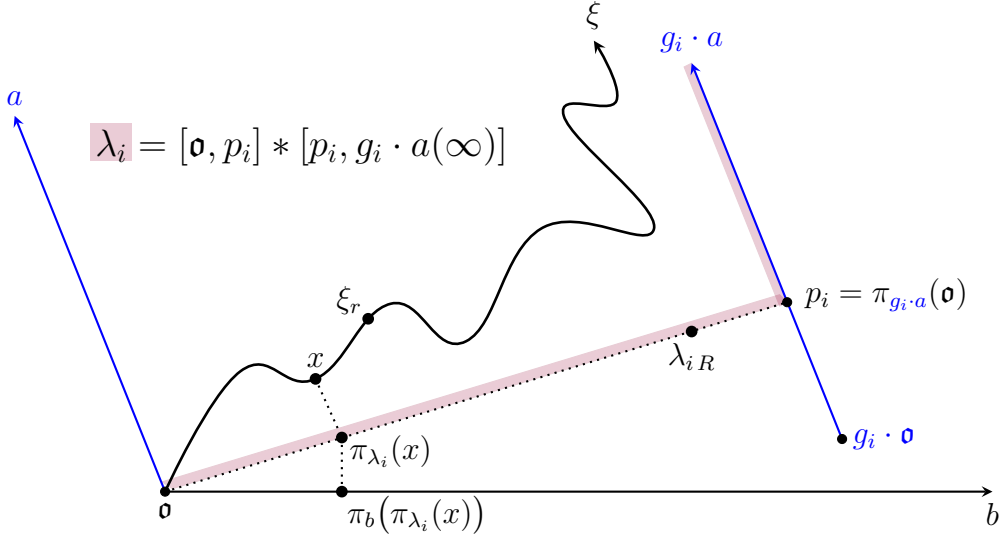
\begin{figure}[H]
        \centering
             \begin{tikzpicture}[scale=1]

%   help lines 
 % \draw[step=1cm,gray, opacity=.5] (1,-2) grid (15,10);
\definecolor{aqua}{rgb}{0.0, 1.0, 1.0}

% \lambda_i
    \draw[black!25!purple, fill = black!20!purple, opacity=0.2  ] plot((.96,-2+.13) -- (1+.1,-2+.04) -- (9.27,.52) -- (9.27-.1,.52+.1) -- cycle;
    \draw[black!25!purple, fill = black!20!purple, opacity=0.2  ] plot((9.27,.52) -- (-1+8+1,3+.7) -- (-1+8+1-.1,3+.7-.05) -- (9.27-.1,.52-.05) -- cycle;

% Legend

     \fill[black!25!purple, fill = black!20!purple, opacity=0.2 ] (7-.2-6-.8,3-.6) rectangle (7.4-.2-6+.15-.8,3.4-.6+.1);
     \node[right,  font=\large] at (6.75+.05-6.07-.8, 3.22-.63) {$ \lambda_i = [\ob, p_i]*[p_i, g_i\cdot a(\infty)]$};

% geodesic a
    \draw [blue,thick, ->](1,-2) to (-1,3);
    \node[blue, above] at (-1,3) {$a$};

%geodesic g_i \cdot a
    \draw [blue,thick, ->](1+8+1,-2+.7) to (-1+8+1,3+.7);
    \node[blue, above] at (-1+8+1,3+.7) {$g_i\cdot a$};
    \node[blue, right] at (1+8+1,-2+.7) {$g_i\cdot \ob$};

% geodesic b
    \draw [thick, ->](1,-2) to (12,-2);
    \node [below] at (12,-2) {$b$};
    \draw[fill] (1,-2) circle [radius=0.06];
    \node[below] at (1,-2) {$\ob$};

% orbit points of the \ob
     \draw[fill] (10,-1.3) circle [radius=0.06];

% Closest point projection of \ob 
    \draw [thick, dotted](1,-2) to (9.27,.52);
    \draw[fill] (9.27,.52) circle [radius=0.06];
    \node [right] at (9.37,.62) {$p_i=\pi_{\color{blue}g_i\cdot a}(\ob)$};

% quasi geodesic \xi

 \pgfsetlinewidth{1pt}
  \pgfsetplottension{.75}
  \pgfplothandlercurveto
  \pgfplotstreamstart
  \pgfplotstreampoint{\pgfpoint{1cm}{-2cm}}  
  \pgfplotstreampoint{\pgfpoint{2cm}{-.4cm}}   
  \pgfplotstreampoint{\pgfpoint{3cm}{-.5cm}}
  \pgfplotstreampoint{\pgfpoint{4cm}{.5cm}}
  \pgfplotstreampoint{\pgfpoint{5cm}{0cm}}
  \pgfplotstreampoint{\pgfpoint{6cm}{1.5cm}}
  \pgfplotstreampoint{\pgfpoint{7.5cm}{1.5cm}}
  \pgfplotstreampoint{\pgfpoint{6.3cm}{3cm}}
  \pgfplotstreampoint{\pgfpoint{7cm}{3.2cm}}
  \pgfplotstreampoint{\pgfpoint{6.7cm}{4cm}}
  \pgfplotstreamend
  \pgfsetarrowsend{>}
  \pgfusepath{stroke} 
  \pgfsetarrowsend{}
  \node [above] at (6.7,4) {$\xi$};

% x \in [\ob, \xi_r] and its projections

\draw[fill] (3.7,.3) circle [radius=0.06];
\node [above] at (3.45,.2) {$\xi_r$};

\draw[fill] (3,-.5) circle [radius=0.06];
\node [above] at (3,-.5) {$x$};

\draw[fill] (3.35,-1.28) circle [radius=0.06];
\node [below] at (4,-1.1) {$\pi_{\lambda_i}(x)$};

\draw [thick, dotted](3,-.5) to (3.35,-1.28);
\draw [thick, dotted](3.35,-1.28) to (3.35,-2);

\draw[fill] (3.35,-2) circle [radius=0.06];
\node [below] at (4.35,-1.9) {$\pi_b\big(\pi_{\lambda_i}(x)\big)$};

\draw[fill] (8,.13) circle [radius=0.06];
\node [below] at (8,.13) {$\lambda_{i\, R}$};

    \end{tikzpicture}
        \caption{A visual for Proposition \ref{Prop: lambda_i is in the open set}. We choose $R$ large enough for any $\xi$ and any $x \in [\ob, \xi_r]_\xi$, its projection to $\lambda_i$ will be within $[\ob, p_i]_{\lambda_i}$. This bounds the distance of all $x \in [\ob, \xi_r]_\xi$ to $b$ in terms of $m_b$ and $m_\lambda$.}
        \label{fig:projection of xi_r}
    \end{figure}

    See Figure \ref{fig:projection of xi_r}. Note that by Lemma \ref{Lem: Lambda is k-Morse}, $2m_b(9, 0)+m_\lambda(\qq, \sQ)$ is also a Morse gauge for $b$.  By \cite[Lemma 4.2]{QRT2} and its proof, there will also exist an $i$ such that any $\xi$ with $\xi \sim \lambda_i$ will also have $\xi|_r \subset \calN_\kappa \big(b, m_b(\qq, \sQ) \big). $
\end{proof}

\begin{corollary} \label{Cor:convergence in k-morse bondary}
    Let $K\geq 0$ be the cobounded constant of $G\curvearrowright X$ and assume $|\partial_\K X | \geq 3$. For any $\Gb \in \partial_\K X $, choose a geodesic $b \in \Gb$ and a sequence $(g_i)_i$ that tracks $b$. Let $\Ga$ be as in Lemma \ref{Lem:g_ia leaves ball of radius R}. Then for any $r>0$, there exists $i$ so that $g_i\Ga\in\mathcal{U}(b,r)$.
\end{corollary}

\begin{proof}
    This is Proposition \ref{Prop: lambda_i is in the open set} rewritten using the definition of $\mathcal{U}(b,r)$.
\end{proof}

We now prove the main result of this section, minimality of the sublinearly Morse boundary. Notice that, by Lemma \ref{Lem:g_ia leaves ball of radius R}, for any $\Gb\in\partial_\kappa X$, there exists some element $\Ga\neq\Gb$ so that $\Gb\in\overline{G\Ga}$. However, we need to show that \textbf{any} element of $\partial_\kappa X$ has a dense orbit. For this, Remark \ref{Remark:Symmetry in construction} will be a helpful observation.

\begin{theorem}[Minimality of $\partial_\kappa G$]\label{density}
   Let $K\geq 0$ be the cobounded constant of $G\curvearrowright X$ and assume $|\partial_\K X | \geq 3$. For every $\Gc\in\partial_\kappa X$, the orbit $G\Gc$ is dense in $\partial_\K X$.
\end{theorem}
\begin{proof}
Let $\Gb,\Gc \in \partial_\kappa X$. If $\Gb=\Gc$, we are done. Otherwise, let $\Ga\neq\Gb\neq\Gc$. Let $a\in\Ga$ $b \in \Gb$, and $c\in\Gc$ be geodesic ray representatives all with domain $[0,\infty)$. Let $\{g_i\}$ be a sequence in $G$ that tracks $b$ and let $\{h_j\}$ be a sequence in $G$ that tracks $a$ as in Lemma \ref{Lem:g_ia leaves ball of radius R}. Let $r>0$ be arbitrary. By Corollary \ref{Cor:convergence in k-morse bondary} there exists $i$ so that
$g_i\Ga\in\mathcal{U}(b,r)$. It is clear that, since the group action of $G\curvearrowright X$ is by isometries, $\Ga\in\mathcal{U}(g_i^{-1}b,r)$. 
By \cite[Claim 4.6]{QRT2}, there exists $r'>0$ and $a'\in\Ga$ so that $\mathcal{U}(a',r')\subseteq \mathcal{U}(g_i^{-1}b,r)$. 
Again by Corollary \ref{Cor:convergence in k-morse bondary} (and keeping in mind Remark \ref{Remark:Symmetry in construction}) there exists $j$ so that $h_j\Gc\in\mathcal{U}(a',r')$, and so $h_j\Gc\in \mathcal{U}(g_i^{-1}b,r)$, i.e., $g_ih_j\Gc\in\mathcal{U}(b,r)$.
\end{proof}

We note that Theorem \ref{thm:minimality of sublinear} adds an additional assumption that the group action $G\curvearrowright X$ is also proper, and in fact is the special case where $X$ is a Cayley Graph for $G$. Therefore, Theorem \ref{thm:minimality of sublinear} is a special case of Theorem \ref{density}.

\subsection{North-South Dynamics on $\partial_K X$}
Similar to the idea of rank-one isometry we can also define $\kappa$-Morse isometry as the an element $g \in G$ whose axis is a $kappa$-Morse geodesic ray from some point onward. Also pertain to the analogy with rank-one isometry we produce the result with the rank-rigidity flavor in Proposition~\ref{Prop: Sublinear Morse Isometry is a Morse isometry}.
\begin{definition}[$\kappa$-Morse isometry]\label{def: Sublinear Morse Isometry}
Let $X$ be a geodesic metric space and let $\ob\in X$ be a basepoint. For an isometry $g: X \rightarrow X$, denote $\ob_n=g^n\left(\ob\right), \eta_i^j=\left[\ob_i, \ob_{i+1}\right] \cup \ldots \cup\left[\ob_{j-1}, \ob_j\right]$, where $i, j \in \mathbb{Z}$. For convenience, we allow that $i=-\infty$ or $j=\infty$. We say the action of $g$ on $X$ is $\kappa$-Morse or the isometry $g$ is $\kappa$-Morse if there exist a $\kappa$-Morse gauge $m(q, Q)$ and a geodesic $\left[\ob_i, \ob_{i+1}\right]$ for each $i$ such that the bi-infinite concatenation of geodesics $ \eta_{-\infty}^{\infty}$ is an $\kappa$-Morse quasi-geodesic.
\end{definition}

Note that, when $\kappa \equiv 1$, this is the definition of a Morse isometry given in \cite{liu2021dynamics}. Using the properties of $\kappa$-Morse quasi-geodesics, it is easy to see that the notion does not depend on the choice of basepoint or choices of geodesics. In particular, we can choose a geodesic $\left[\ob, \ob_1\right]$ and take $\left[\ob_i, \ob_{i+1}\right]=g^i\left[\ob, \ob_1\right]$ for all $i$. This will be our construction in the following Proposition.

\begin{proposition}\label{Prop: Sublinear Morse Isometry is a Morse isometry}
    If $g \in G$ is a sublinear Morse isometry, then $g$ is a Morse isometry.
\end{proposition}

\begin{proof}
    Let $g$ be a sublinear Morse isometry. Then there exists a concatenation of geodesics $\gamma = \eta_{-\infty}^\infty$ as in Definition \ref{def: Sublinear Morse Isometry}. We define $$N(\qq, \sQ) = \text{sup}\bigg\{d(\gamma, x) : x\in  \calN_\kappa \big(\gamma, m_\gamma(\qq, \sQ) \big) \text{ and } \exists x_\gamma \in \pi_\gamma(x) \text{ with } x_\gamma \in [\ob,\ob_1]\bigg\}.$$ We see that since the points in $\calN_\kappa \big(\gamma, m_\gamma(\qq, \sQ) \big)$ that closest point project to $[\ob, \ob_1]$ are bounded, we get that $N(\qq, \sQ)$ will be finite for all $(\qq, \sQ)$.

    Now, choose any quasi-geodesic $\xi$ with endpoints on $\gamma$. Pick any point $x \in \xi $, and let $x_\gamma \in \pi_\gamma(x).$ We have $x_\gamma \in g^i[\ob, \ob_1]$ for some $i.$ So, $g^{-i}x_\gamma \in [\ob, \ob_1].$ Since $g$ is an isometry, $g^{-i}\cdot \xi$ is still a $(\qq, \sQ)$-quasi geodesic with endpoints on $\gamma$ and $g^{-i} x_\gamma \in g^{-i}\cdot\xi$ with $g^{-i} x_\gamma \in \pi_\gamma( g^{-i}\cdot x)$. Thus, by how $N(\qq, \sQ)$ is defined,  $$d(g^{-i}x, \gamma) \leq N(\qq, \sQ).$$ Since $g$ is an isometry that preserves $\gamma$, $$d(x, \gamma) \leq N(\qq, \sQ).$$ Because this is true for all $x \in \xi$ , we conclude that $\gamma$ is $N$-Morse. See Figure \ref{fig:sublinear isometry}. \end{proof}

     \begin{figure}[H]
    \centering
    \begin{tikzpicture}[scale=1]

%   help lines 
 % \draw[step=1cm,gray, opacity=.5] (-6,0) grid (6,5);
\definecolor{aqua}{rgb}{0.0, 1.0, 1.0}

% geodesic gamma
    \draw [thick, <->](-6.2,0) to (6.2,0);
    \node [below] at (-6,-0.1) {$\gamma$};
    \draw[fill] (0,0) circle [radius=0.06];
    \node[below] at (0,0) {$\ob$};

% orbit points of the \ob
     \draw[fill] (1,0) circle [radius=0.04];
     \draw[fill] (5,0) circle [radius=0.04];
     \draw[fill] (6,0) circle [radius=0.04];

\node [below] at (1,0) {$\ob_1$};
\node [below] at (5,0) {$\ob_i$};
\node [below] at (6,0) {$\ob_{i+1}$};

\node [below] at (2.3,0.5) {$\cdots$};

% quasi geodesic
     \color{red}
 \pgfsetlinewidth{1pt}
  \pgfsetplottension{.75}
  \pgfplothandlercurveto
  \pgfplotstreamstart
  \pgfplotstreampoint{\pgfpoint{3.6cm}{0cm}}  
  \pgfplotstreampoint{\pgfpoint{3.9cm}{.6cm}}   
  \pgfplotstreampoint{\pgfpoint{4.2cm}{.3cm}}
  \pgfplotstreampoint{\pgfpoint{4.5cm}{.9cm}}
  \pgfplotstreampoint{\pgfpoint{4.8cm}{.6cm}}
  \pgfplotstreampoint{\pgfpoint{5cm}{.8cm}}
  \pgfplotstreampoint{\pgfpoint{5.3cm}{.4cm}}
  \pgfplotstreampoint{\pgfpoint{5.6cm}{.5cm}}
  \pgfplotstreampoint{\pgfpoint{5.9cm}{0cm}}
  \pgfplotstreamend
  \pgfusepath{stroke} 
  \node [above] at (3.9,.6) {$\xi$};
  \color{black}
 
  \draw[fill] (5.6,.5) circle [radius=0.04];
  \node [above] at (5.65,.5) {$x$};
  \draw [thick, dotted](5.6,.5) to (5.6,0);

% translated quasi geodesic
 \color{red}
 \pgfsetlinewidth{1pt}
  \pgfsetplottension{.75}
  \pgfplothandlercurveto
  \pgfplotstreamstart
  \pgfplotstreampoint{\pgfpoint{3.6cm-5cm}{0cm}}  
  \pgfplotstreampoint{\pgfpoint{3.9cm-5cm}{.6cm}}   
  \pgfplotstreampoint{\pgfpoint{4.2cm-5cm}{.3cm}}
  \pgfplotstreampoint{\pgfpoint{4.5cm-5cm}{.9cm}}
  \pgfplotstreampoint{\pgfpoint{4.8cm-5cm}{.6cm}}
  \pgfplotstreampoint{\pgfpoint{5cm-5cm}{.8cm}}
  \pgfplotstreampoint{\pgfpoint{5.3cm-5cm}{.4cm}}
  \pgfplotstreampoint{\pgfpoint{5.6cm-5cm}{.5cm}}
  \pgfplotstreampoint{\pgfpoint{5.9cm-5cm}{0cm}}
  \pgfplotstreamend
  \pgfusepath{stroke} 
  \node [above] at (-1.1,.6) {$g^{-i}\xi$};
  \color{black}

  \draw[fill] (5.6-5,.5) circle [radius=0.04];
  \node [above] at (5.85-5,.5) {$g^{-i} x$};
  \draw [thick, dotted](5.6-5,.5) to (5.6-5,0);

 \draw [dashed] (0, 1) to [bend left = 10] (6,3){};
 \draw [dashed] (0, -1) to [bend right = 10] (6,-3){};
 \draw [dashed] (0, 1) to [bend right = 10] (-6,3){};
 \draw [dashed] (0, -1) to [bend left = 10] (-6,-3){};

 % shaded region

\fill[blue, opacity=0.2] (0,1) .. controls(.3,1.2) and (.7,1.4) .. (1, 1.55) -- (1,-1.55) ..  controls(.7,-1.4) and (.3,-1.2) ..  (0,-1) -- cycle;

    \end{tikzpicture}
    \caption{A picture of the argument in \ref{Prop: Sublinear Morse Isometry is a Morse isometry}. Note that $N(\qq,\sQ)$ is largest distance to $\gamma$ among points in the blue region.}
    \label{fig:sublinear isometry}
\end{figure}
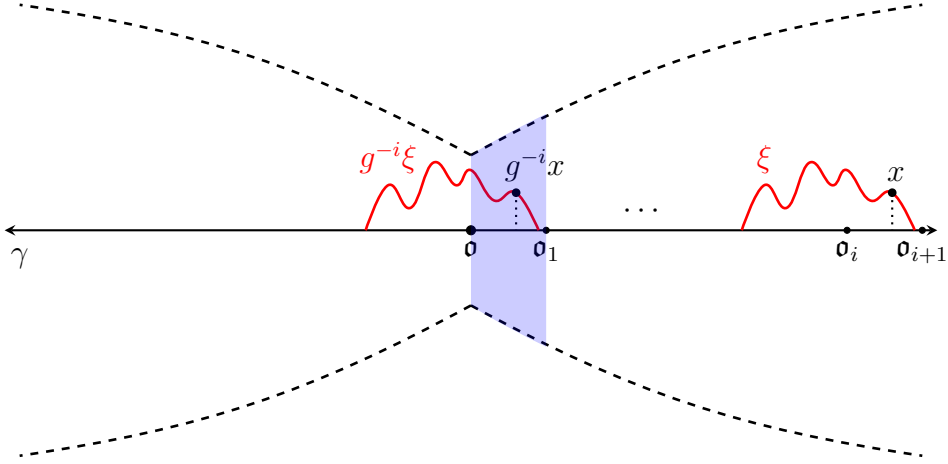

As an immediate corollary, we recover weak north-south dynamics on the sublinearly Morse boundary. This is a direct analogue to the statement of weak north-south dynamics in the Morse boundary, see \cite[Corollary 6.8]{liu2021dynamics}.

\begin{corollary}[Weak North-South Dynamics for sublinearly Morse Boundaries]
    Let $X$ be a proper geodesic space and $\ob$ be a basepoint. Let $g$ be a Sublinearly Morse isometry of $X$. Given any open neighborhood $U$ of $g^{\infty}$in $\partial_\kappa X$ and any compact set $K \subset \partial_\kappa X$ with $g^{-\infty} \notin K$. There exists an integer $n$ such that $g^n(K) \subset U$.
\end{corollary}

\begin{proof}
    This is immediate from Proposition \ref{Prop: Sublinear Morse Isometry is a Morse isometry}, Theorem \ref{density}, and the proof of \cite[Theorem 3.6]{QZ}.\end{proof}

\section{Topological properties of the quasi-redirecting boundary}

We now turn our attention to quasi-redirecting boundaries. As discussed in \cite{QR24}, a main motivator for the QR boundary is to be a bordification which is as large as possible. In this next theorem, we show that QR boundaries are not too large compared to the sublinearly Morse boundary. In particular, we show that when the QR boundary of $G$ is defined (as a topological space), as long as the sublinearly Morse boundary is sufficiently large, (i.e., contains at least three elements,) then it is dense in the QR boundary. We assume in this section that the QR boundary exists for the groups in equestion.

\begin{theorem}\label{Theorem: Miminality QR}

Assume that $\partial G$ exists. If $|\partial_\kappa G|\geq 3$, then $\partial_\kappa G$ is a dense subset of the QR boundary. In particular, if $\bfb$ is any element of $\partial G$, then there exists an infinite sequence $\{\bfa_i\} \in \pka G$ such that the sequence converges to $\bfb$ in the topology of $\partial G$. 
\end{theorem}

\begin{proof}
 Let $X$ be a Cayley graph of $G$ (with respect to a finite generating set) once and for all. Consider a geodesic representative $b \in \bfb \in \partial X$. By Lemma \ref{Lem:(81,2K)-quasi-geo-ray}, there exists $a\in\bfa$ so that for any radius $R$ there exists an element $g$ so that $[\ob, p]\cup[p,g\cdot a(\infty)]$ quasi-redirects to $b$ at radius $R$ via an $(27,3K)$ quasi-geodesic, where $p=\pi_{g\alpha}(\ob)$.

Let $r > 0$. Then by Lemma \ref{Lem:Max}, there exists $\qq_{\rm max}$ such that $\qq \nleq \qq_{\rm max}$ implies any $\qq$-ray $\alpha$ quasi-redirects the $b$ at radius $r$. Denote $m_a$ as the $\K$-Morse gauge for $a$ (we can assume that $m_a(100\qq_{\rm max})$ is also small compared to $r$). Let $\{g_i\}$ be a group sequence that tracks $b$ and let $\bfa_i = g_i\cdot \bfa$. By Lemma \ref{Lem:g_ia leaves ball of radius R}, there exists an $n$ so that, for some $\K$-Morse geodesic $a \in \bfa$,  $g_n\cdot a \cap B_{10m_a(\qq_{\rm max}) +r +2K}(\ob) = \emptyset$. Note that, by the first paragraph, $[\ob, p]\cup[p,g\cdot a(\infty)]$ quasi-redirects to $b$ at radius $r$.

Let $\alpha \in \bfa_n \in \partial_\K X$ be a $\qq = (q,Q)$ quasi geodesic so that $\qq \leq \qq_{\rm max}$. Consider $\ob_n = g_n\cdot \ob$ and $\pi_\alpha(\ob_n)$. Because $[\ob_n, \pi_\alpha(\ob_n)]\cup [\pi_\alpha(\ob_n), \alpha(\infty)]_\alpha$ is a $(3q, Q)$-quasi-geodesic that fellow travels $g_n\cdot a$, we must have that $\pi_\alpha(\ob_n)$ is within $m_a(\qq_{\rm max})$ of $g_n \cdot a$. In particular, $\pi_\alpha(\ob_n)$ is not contained in $\alpha|_r$. This means that $[\ob, \pi_\alpha(\ob_n)]_\alpha \cup [\pi_\alpha(\ob_n), \ob_n]$ is a $(3q,Q)$ quasi geodesic that contains $\alpha|_r$ and has $\ob_n$ within $K$ of $b_{R'}$ for some $R'$.  

Using some sufficiently large $s>0$ one can now perform a segment to geodesic ray surgery (Lemma \ref{surgery}) on $[\ob, \pi_\alpha(\ob_n)]_\alpha \cup [\pi_\alpha(\ob_n), \ob_n]$ and $[b_s, b(\infty)]_b$. Also, since $\ob_n$ is within $K$ of $b_{R'}$, it must be that the point on $[\ob, \pi_\alpha(\ob)]_\alpha \cup [\pi_\alpha(\ob), \ob_n]$ that realizes the set distance to $[b_s, b(\infty)]_b$ is not contained in $\alpha|_r$. Thus, We have found a $(12q,3Q)$-quasi-geodesic that redirects $\alpha$ to $b$ at radius $r$. See Figure \ref{fig: Minimality QR}. 

Therefore by Definition \ref{Def:The-In-Topology}, for every $r>0$ there exists $g_n\in G$ so that $g\cdot\bfa=\bfa_n\in\mathcal{U}(\bfb,r)$, and so $\bfb$ is in the closure of $G\bfa$, i.e. $\bfb\in\overline{G\bfa}$.

\end{proof}
\begin{figure}[h]
    \centering

 \begin{tikzpicture}[scale=.8]

%   help lines 
 % \draw[step=1cm,gray, opacity=.5] (1,-2) grid (15,10);
\definecolor{aqua}{rgb}{0.0, 1.0, 1.0}

% geodesic a
    \draw [blue,thick, ->](1,-2) to (-1,3);
    \node[blue, above] at (-1,3) {$a$};

%geodesic g_n \cdot a
    \draw [blue,thick, ->](1+8+1,-2+.7) to (-1+8+1,3+.7);
    \node[blue, above] at (-1+8+1,3+.7) {$g_n\cdot a$};
    \node[blue, left] at (1+8+1,-2+.6) {$\ob_n$};

% geodesic b
    \draw [thick](1,-2) to (15,-2);
    \node [below] at (16,-2) {$b$};
    \draw[fill] (1,-2) circle [radius=0.06];
    \node[below] at (1,-2) {$\ob$};
    \node[below] at (10,-2) {$b_{R'}$};
    \draw[fill] (10,-2) circle [radius=0.06];
    \draw[fill] (15,-2) circle [radius=0.06];
    \node[below] at (15,-2) {$b_{s}$};
    \draw [thick, red, ->](15,-2) to (16,-2);

% Projection of \ob_n to \alpha

\draw [thick, dotted, red](1+8+1,-2+.7) to (6,1.5);
\draw[fill] (1+8+1,-2+.7) circle [radius=0.06];
\draw[fill] (6,1.5) circle [radius=0.06];
\node[above] at (5.5,1.5) {$\pi_\alpha(\ob_n)$};

% Set distance between [b_s, b(\infty)] and [\ob, \pi_\alpha(\ob)]_\alpha \cup [\pi_\alpha(\ob), \ob_n]

\draw [thick, dotted, red](1+8+1,-2+.7) to (15,-2);

  % quasi geodesic \xi

  \node [above] at (6.7,4) {$\alpha$};

  \pgfsetlinewidth{1pt}
  \pgfsetplottension{.75}
  \pgfplothandlercurveto
  \pgfplotstreamstart
  \pgfplotstreampoint{\pgfpoint{1cm}{-2cm}}  
  \pgfplotstreampoint{\pgfpoint{2cm}{0cm}}   
  \pgfplotstreampoint{\pgfpoint{3cm}{-.5cm}}
  \pgfplotstreampoint{\pgfpoint{4cm}{.5cm}}
  \pgfplotstreampoint{\pgfpoint{5cm}{0cm}}
   \pgfplotstreampoint{\pgfpoint{5.5cm}{1cm}}
  \pgfplotstreampoint{\pgfpoint{6cm}{1.5cm}}
  \pgfplotstreamend
  \color{red}
  \pgfusepath{stroke}

 \pgfsetlinewidth{1pt}
  \pgfsetplottension{.75}
  \pgfplothandlercurveto
  \pgfplotstreamstart
  \pgfplotstreampoint{\pgfpoint{6cm}{1.5cm}}
  \pgfplotstreampoint{\pgfpoint{6.5cm}{1.8cm}}
  \pgfplotstreampoint{\pgfpoint{7.5cm}{1.8cm}}
  \pgfplotstreampoint{\pgfpoint{6.3cm}{3cm}}
  \pgfplotstreampoint{\pgfpoint{7cm}{3.2cm}}
  \pgfplotstreampoint{\pgfpoint{6.7cm}{4cm}}
  \pgfplotstreamend
  \pgfsetarrowsend{>}
  \color{black}
  \pgfusepath{stroke} 

% Circle of radius r
\draw[thin, dotted] (1,-2) circle [radius=2];
\draw[fill] (1.83,-.2) circle [radius=0.06];
\node[above] at (1.78,-.13) {$\alpha_r$};

    \end{tikzpicture}
    \caption{A visual for the proof of Theorem \ref{Theorem: Miminality QR}. We choose $g_n$ so that $g_n\cdot a$ is sufficiently far away from $\ob$. Then we use the segment to geodesic ray surgery lemma to create a $(12q,3Q)$-quasi-geodesic (in red) comprised of connecting $[\ob, \pi_\alpha(\ob_n)]_\alpha \cup [\pi_\alpha(\ob_n), \ob_n]$ and $[b_s, b(\infty)]$. }
    \label{fig: Minimality QR}
\end{figure}
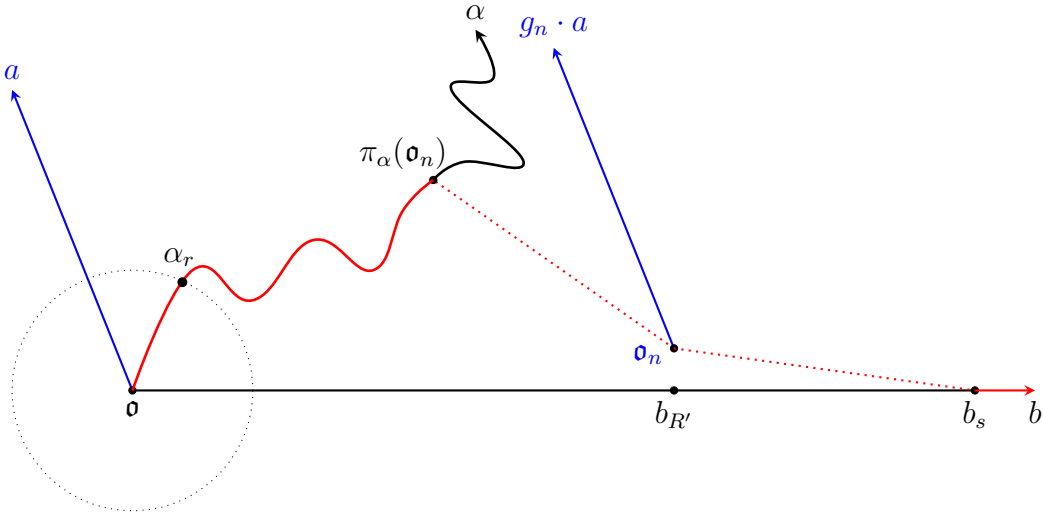

As an immediate corollary of this theorem, we can show that the QR-boundary of $G$, when exists, is second countable, if the sublinearly Morse boundary contains 3 or more points. Second countability of $\partial X$ is shown for asymptotically tree-graded spaces by Rafi and Qing in \cite[Proposition 8.16]{QR24}. Thus the corollary greatly expand the set of groups whose QR boundary (if exists) is second countable.

\begin{corollary}[Second countability]
  Assume that $\partial G$ exists.  If $|\pka G|\geq 3$, then $\partial G$ is second countable. 
\end{corollary}

\begin{proof}
    Choose $\Ga\in\partial_\kappa X$ so that $G\Ga$ is dense in $\partial X$. Since $G$ is finitely generated, $G\Ga$ is countable. Then take all open sets $\mathcal{U}(g\Ga,r)$ as in 
    Definition~\ref{Def:The-In-Topology} with $r\in\mathbb{Q}^+$.
\end{proof}

\subsection{The $G$ action is not minimal.}
\label{counter}
On the other hand, Section 11 of \cite{QR24} exhibit an example where the action of $G$ on $\partial X$ is not minimal. We briefly describe the example here and refer to
\cite[Section 10]{QR24} for a complete treatment. Let $G$ be the following irreducible right-angled Artin group: 
\[
G = \big \langle a, b, c, d | [a,b], [b, c], [c, d]\big \rangle.
\]

A \emph{block} in $X$ is a convex infinite subset of $X$ that is a lift of either  $S_1 \cup S_2$, 
or $S_2 \cup S_3$. Thus a block is isometric to the universal cover of the Salvetti complex of either of the following groups 
\[ 
G_1 = \big\langle a, b, c | [a, b], [b, c] \, \big\rangle \qquad \text{or} \qquad 
G_2 =\big \langle b, c, d | [b, c], [c, d] \, \big\rangle. 
\]
In other words, each blocks comes with a co-compact action of a conjugate copy of either $G_1$ of $G_2$. 
A \emph{flat} in $X$ is a lift of $S_1, S_2$ or $S_3$. Given a pair of blocks, their intersection in $X$ is 
either empty or a \emph{flat} (in fact, always a lift of $S_2$). That is, a flat comes with a compact action of 
a conjugate of the group $ \big\langle b, c | [b, c] \, \big\rangle = \ZZ^2$.  We refer to these as $bc$--flats. 

One can construct a graph where vertices are blocks and two vertices are 
connected by an edge if and only if two blocks intersect in a flat. The resulting graph, which we denote by $T$, 
is the Bass-Serre tree associated with amalgamated product decomposition of $G$:
\[
G = G_1 *_{ \langle \, b, c \, | \, [b, c] \, \rangle} G_2.
\]

The quasi-redirecting boundary of $G$ consists of the following classes. We say a quasi-geodesic ray is \emph{transient} if its projection to $T$ departs every ball of bounded radius in $T$. To every transient quasi-geodesic ray $\alpha$, we can associate
an itinerary $A_i$, which is an infinite embedded path in $T$ exiting an end $\xi$. Given such $\xi$, 
there is a preferred quasi-geodesic $\alpha_1$ exiting $\xi$ that passes through the points $x_k$. 
We set $w_k = x_{k+1} \cdot x_k^{-1}$ and refer to $|w_k|$ as the excursion of $\alpha_1$ in the block $A_i$. 
Then $[\alpha_1]$ is different from $\bfz$ if and only if the excursion of $\alpha_1$ is sub-exponential with respect to the distance in $T$. That is, every class in $P(X)$ is either $\bfz$ or $[\alpha_1]$ for such $\alpha_1$. Lastly, by \cite[Corollary 11.11]{QR24}, two quasi-geodesic rays with sub-exponential excursion, and with different itineraries, are in different, unrelated classes. 
\begin{proposition}
Let $G$ be the group specified above. The set $G \cdot \bfz$ is not a dense subset of $\partial G$.
\end{proposition}
\begin{proof}
Note that $\bfz$ is the only element in $\partial G$ that contains a geodesic ray that bounds a flat. Any $g \in G$ acts by isometry on $G$ and thus $g \cdot \bfz$ also bounds a flat. But $\bfz$ is the only element with such property thus $g \cdot \bfz = \bfz$ for any $g \in G$. Thus if there exists a sequence $\{ g_i \cdot \bfz\}$ that converges to $\bfa$ for each $\bfa \in \partial G$, then $\bfz$ is in every neighborhood of $\bfa$ and thus $bfz \preceq \bfa$. Choose $\bfa$ to be a Morse element and by \cite[Proposition 6.5]{QR24}, $\bfz \sim \bfa$ and $\bfz$ is also a Morse element, contradicting choice of $\bfz$.
\end{proof}
\begin{remark}
In this example, even though the cardinality of sublinearly Morse element is uncountable, the cardinality of \[ P(G)\setminus \pka G\] is also uncountable. That is to say, the quasi-redirecting boundary is much larger as a set than the sublinearly Morse boundary.
\end{remark}

\section{Existence of Morse elements}
Question 4.4 in \cite{QR24} asks if $G$ does not have an Morse element, is $P(G)$ a single point. In this section we answer the question in the affirmative for finite dimensional \CAT cube complexes.
\subsection{\CAT cube complexes and rank rigidity theorem}
A \emph{\CAT cube complex} is a simply connected cell complex where all of the cells are Euclidean cubes with edge length one, and such that the link of each vertex is a flag complex. A theorem of M. Gromov (see e.g. Theorem II.5.20 from \cite{BH1} for the finite-dimensional case and Theorem 40 from \cite{Lea} for the general case) states that $X$, endowed with the induced length metric, is a \CAT space. Moreover $X$ is \emph{complete} if and only if $X$ does not contain an infinite, ascending chain of cells (see \cite[ Theorem 31]{Lea}); in particular, this is the case if $X$ is locally finite-dimensional, in the sense that the supremum of the dimensions of cubes containing any given vertex is finite. In this section we are only concerned with finial dimensional, proper, \CAT cube complexes. The space $X$ is locally compact if and only if it is locally finite, in the sense that every vertex has finitely many neighbours.

An open conjecture due to Ballmann and Buyalo \cite{BB08}  classifies \CAT spaces by means of their rank. In particular, a complete geodesic is said to have \emph{rank 1} if it does not bound a flat half-plane. Accordingly, a \CAT space is called \emph{rank-one} if it contains a rank 1 geodesic, otherwise we say it has higher rank.
\begin{conjecture}
Let X be a locally compact \CAT space and let $G$ acts geometrically on $X$. If X contains a geodesic of rank 1, then it also contains a $G$--periodic geodesic of rank 1.
\end{conjecture}

There are various results addressing the conjecture by adding extra conditions that support the conclusion, among the most famous is the following:

\begin{theorem}[Rank Rigidity of finite-dimensional CAT(0) cube complexes]\cite{CS11} Let X be a finite-dimensional \CAT cube complex and $G \leq Aut(X)$ be a group acting geometrically on $X$. Then there is a convex $G$-invariant subcomplex $Y \subseteq X$ such that either Y is a product of two unbounded cube subcomplexes or $G$ contains an element acting on $\gamma$ as a rank-one isometry.
\end{theorem}

In this section, we will use the rank-rigidity theorem to obtain rank-one isomoetry in the existance of nontrivial QR-boundary. We first present a lemma that helps to establish when two geodesic rays are QR-equivalent. Assume for the rest of the this section that the QR boundary is well defined for the  \CAT cube complex in question.
\begin{lemma}\label{flatredirecting}
Let $X$ be a finite dimensional \CAT cube complex and let $\alpha, \beta$ be two geodesic rays emanating from the basepoint $\go$. Suppose there exists a $C>0$ such that the geodesic segments $[\alpha(n), \beta(n)]$ are such that there are infinitely many $n \in \NN$ where
\[ d(\go, [\alpha(n), \beta(n)]) \geq C \cdot n,\]
Then $\alpha \sim \beta$.
\end{lemma}
\begin{proof}
By the convexity of \CAT geometry there exists a point $p \in [\go, \alpha(n)]$ and a point $q \in  [\alpha(n), \beta(n)]$ such that  
\[ d(p, \go) = \frac 12 n \text{ and } d(p, q) = d(p, [\alpha(n), \beta(n)] )\leq \frac 12 C \cdot n.\]
In fact let $p$ be a point that realizes the set distance between $[\go, p]$ and $[\alpha(n), \beta(n)]$. This is possible by the convexity of \CAT geometry. 

By Lemma \ref{surgery}, surgery I, the concatenation $[p, q] \cup [q, \beta(n)]$ is a $(3,0)$--geodesic segment. Furthermore, let $D =d(p ,q) =  d(p, [\alpha(n), \beta(n)]$. There exists a point $p' \in  [\go, \alpha(n)]$ and a point $q' \in [p, q]$ such that 
\[
d(p', q') = d(q',  [\alpha(n), \beta(n)]) = \frac 12 D = d(p', [q', q]\cup [q, \beta(n)].
\]
Likewise the segment $[p', q']$ realizes the set distance between $[\go, p']$ and $[q', q]$. Again by Lemma \ref{surgery}, surgery I, the concatenation 
\[
\ell : =[p',q'] \cup [[q', q]] \cup [q, \beta(n)]
\]
is a $(9,0)$-quasi-geodesic segment. Furthermore, consider the concatenation
\[
[\go, p'] \cup \ell.
\]
Let $x \in [\go, p']$ and $y \in [p', q']$, by construction $[\go, p'] \cup [p', q']$ is a $(3,0)$-quasi-geodesic segment. Suppose $x \in [\go, p']$ and $y \in [q', q]$. Then by  Lemma \ref{surgery}, surgery I again since $[q', q]$ is the path of nearest-point projection from $q$ to  $[\go, p'] \cup [p', q']$, we have that 
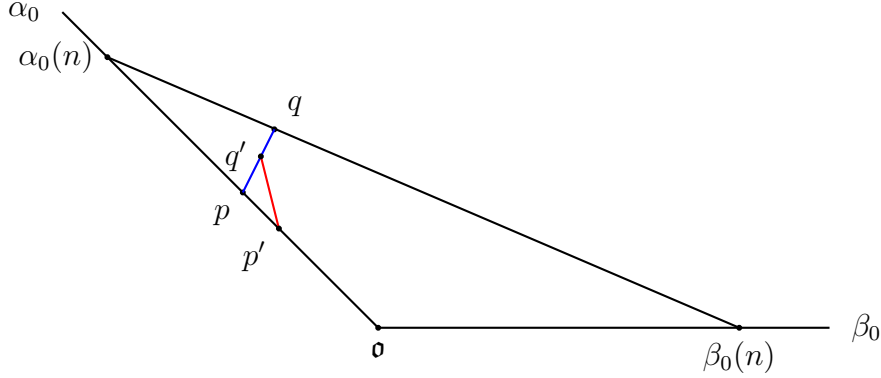
\[
[\go, p'] \cup [p', q'] \cup [q' q]
\]
is a $(9,0)$-quasi-geodesic segment. Lastly suppose $x \in [\go, p']$ and $y \in [q, \beta(n)]$, then 
\[
d(x, y) \geq d(p, q) \geq D \geq n
\]
On the other hand, the arc-length between $x$ and $y$ is bounded above by the arclength of $[\go, p'] \cup \ell$, which is less than $3n$. Therefore the quasi-geodesic constants for any pair of such points is bounded above by $(3,0)$. Thus all case considered The segment $[\go, p'] \cup \ell$ is a $(9,0)$-quasi-geodesic segment.
Lastly, apply Lemma \ref{surgery}, surgery III, to the segment $[\go, p'] \cup \ell$ and the ray $\beta(n, \infty)$, we get a $(36,0)$-quasi-geodesic ray that redirects $[\go, \alpha (\frac{n}{2})]$ to the tail of $\beta$. Since this holds for $n$ we have that $\alpha \preceq \beta$. The symmetric argument shows that $\beta \preceq \alpha$ and thus $\alpha \sim \beta$. 
\end{proof}

\begin{figure}
\begin{tikzpicture}[scale=1.2]
 \tikzstyle{vertex} =[circle,draw,fill=black,thick, inner sep=0pt,minimum size=.5 mm]
[thick, 
    scale=1,
    vertex/.style={circle,draw,fill=black,thick,
                   inner sep=0pt,minimum size= .5 mm},
                  
      trans/.style={thick,->, shorten >=6pt,shorten <=6pt,>=stealth},
   ]

  \node[vertex] (z) at (0,0)[label=below:$\go$]{}; 
  \draw[thick](0,0)--(-3.5, 3.5){};
   \node at (-3.5, 3.5) [label=left: $\alpha_0$]{};
  
\draw[thick](0,0)--(5,0){};
 \node at (5, 0)[label=right:$\beta_0$]{};
 \node [vertex] at (-3, 3)[label=left:$\alpha_0(n)$]{};
  \node [vertex] at (4, 0)[label=below:$\beta_0(n)$]{};
  \draw[thick](-3, 3)--(4,0){};
   \node [vertex] at (-1.5, 1.5)[label=below left:$p$]{};
    \node[vertex]  at (-1.15,2.2)[label=above right:$q$]{};
  \draw [thick, blue] (-1.5,1.5)--(-1.15,2.2){};
   \draw [thick, red] (-1.1,1.1)--(-1.3,1.9){};
    \node[vertex]  at (-1.1, 1.1)[label=below left:$p'$]{};
     \node[vertex]  at (-1.3,1.9)[label=left:$q'$]{};
  \end{tikzpicture}
\end{figure}

\begin{theorem}\label{T:morse}
Let $G$ be a finitely generated group acting geometrically on $X$ where $X$ is a \CAT cube complex.  If $\partial X$ exists and $|\partial X| \geq 2$, then there exists a Morse element in $G$.
\end{theorem}
\begin{proof}
Let $\bfa \neq \bfb \in P(X)$ be two equivalence classes in $\partial X$. Let $\alpha_0 \in \bfa$, $\beta_0 \in \bfb$ be geodesic representatives. Consider the geodesic segments connecting pairs of points $[\alpha_0(i), \beta_0(i)], i =1,2,3...$. If an infinite sub-sequence of $\{[\alpha_0(i), \beta_0(i)] \}$ intersect a ball of bounded radius, then by Arzel\`a-Ascoli Theorem, the sequence $\{[\alpha_0(i), \beta_0(i)] \}$ converges to a bi-infinite geodesic ray $\gamma_0$. If $\gamma_0$ is not rank-one, than it bounds a flat  and thus the two half lines of $\gamma_0$ are equivalent to each other. However, the two halves of $\gamma_0$ each converges to the $\alpha_0$ and $\beta_0$, thus $\alpha_0 \sim \beta_0$ contradicting to our assumption. Thus $\gamma_0$ is necessarily rank-one. By the Rank-rigidity Theorem, there exists a rank-one element in $G$, which is necessarily Morse.
Otherwise,  let us assume that there exists a $C > 0$ such that 
\[d(\go,[\alpha_0(i), \beta_0(i)] \geq C \cdot i. \]
Lemma~\ref{flatredirecting} shows that $\alpha_0 \sim \beta_0$, contradicting our assumption. 

Therefore the only remaining possibility is that $\alpha_0 \nsim \beta_0$ implies for all subsequences of $\{[\alpha_0(i), \beta_0(i)] \}$, the distances
$d(\go,[\alpha_0(i), \beta_0(i)]$ grow at most sublinearly with respect to $i$. In fact, for any $x \in X$, consider the distances $d(\go, [x,\beta_0(i)]$. If for any infinite sequence of $x_j$ such that $\Norm{x_j} \to \infty$,  we get $d(\go, [x_j,\beta_0(i)]$ is at least linear in $i, j$, then the points $x_j$ lies on a ray that is in the same QR-class as $\beta_0$. Thus for all points $x$ not on a ray that is in a same QR-class as $\beta_0$, suppose $d(\go, [x_j,\beta_0(i)]$ grows at most sublinearly, then by $\kappa$-slim condition, $\beta$ is $\kappa$-Morse. Since $\kappa$-Morse rays are rank-one and by Rank-rigidity Theorem again there exists a Morse element in $G$.
\end{proof}

Lastly recall we say $P(X)$ is a \emph{visibility space}  for any two points $\bfa, \bfb \in P(X)$, then there exists a geodesic $\gamma$ with $\gamma(\infty) \in \bfa$ and $\gamma(-\infty) \in \bfb$. 

\begin{corollary}
Let $X$ be a proper \CAT metric space with a cocompact action. Then $P(X)$ is a visibility space.
\end{corollary}
\begin{proof}
Let $X$ be a  proper \CAT metric space and let $\bfa \nsim \bfb $ be two equivalence classes in $P(X)$. As argued in Theorem~\ref{T:morse}, there are two possible cases to the limit of geodesic lines $[\alpha_0(n), \beta_0(n)]$:
\begin{enumerate}
\item an infinite subsequence of them do not exceed a bounded ball, in this case there exists a bi-infinite geodesic line;
\item  all subsequences move away from $\go$ at most sublinearly in $n$ as $n \to \infty$. In this case $\beta_0$ is a sublinearly Morse ray and by \cite{Z} there exists a bi-infinite geodesic.
\end{enumerate} 
Note that Definition~\ref{Def:Redirection} does not require quasi-geodesic rays to start at the same point of origin. Thus it follows from construction that $\gamma(\infty) \in \bfa$ and $\gamma(-\infty) \in \bfb$.  Therefore $P(X)$ is a visibility space.
\end{proof}

%\bibliography{Citations}{}
\bibliographystyle{alpha}

\end{document}